\documentclass{amsart}
\usepackage[all]{xy}
\usepackage{amssymb}
\usepackage{amsfonts}

\newcommand{\D}{\mathbb{D}}
\newcommand{\F}{\mathbb{F}}
\newcommand{\N}{\mathbb{N}}
\newcommand{\Q}{\mathbb{Q}}
\newcommand{\Z}{\mathbb{Z}}
\newcommand{\p}{{}^p\!}
\newcommand{\pl}{{}_p}
\newcommand{\q}{{}_q}
\newcommand{\s}{{}^s}
\newcommand{\uq}{{}^q}

\newcommand{\barp}{{}^{\bar p}\!}
\newcommand{\barq}{{}^{\bar q}}
\newcommand{\lbarq}{{}_{\bar q}}
\newcommand{\cheq}{{}_{\hat q}}

\newcommand{\dualq}{\hat q}
\newcommand{\ufl}{\mathchoice{{}^\flat}{{}^\flat}{{}^\flat}{\flat}}
\newcommand{\ush}{\mathchoice{{}^\sharp}{{}^\sharp}{{}^\sharp}{\sharp}}
\newcommand{\flatqu}{{}^{\flat q}}
\newcommand{\flatql}{{}_{\flat q}}
\newcommand{\sharpqu}{{}^{\sharp q}}
\newcommand{\bz}{\boldsymbol{0}}
\newcommand{\lbz}{{}_{\bz}}
\newcommand{\barGx}{\overline{Gx}}
\newcommand{\barGy}{\overline{Gy}}

\newcommand{\fD}{\mathfrak{D}}
\newcommand{\fC}{\mathfrak{C}}

\newcommand{\cA}{\mathcal{A}}
\newcommand{\cB}{\mathcal{B}}
\newcommand{\cC}{\mathcal{C}}
\newcommand{\cD}{\mathcal{D}}
\newcommand{\cF}{\mathcal{F}}
\newcommand{\cG}{\mathcal{G}}
\newcommand{\cH}{\mathcal{H}}
\newcommand{\cI}{\mathcal{I}}

\newcommand{\cL}{\mathcal{L}}
\newcommand{\cM}{\mathcal{M}}

\newcommand{\cO}{\mathcal{O}}

\newcommand{\cQ}{\mathcal{Q}}

\newcommand{\cIC}{\mathcal{IC}}

\newcommand{\sD}{{}^s\!\fD}

\newcommand{\cg}[1]{\cC_G(#1)}
\newcommand{\cgl}[2]{\cC_G(#1)_{\le #2}}
\newcommand{\cgg}[2]{\cC_G(#1)_{\ge #2}}

\newcommand{\tcgg}[2]{\tilde\cC_G(#1)_{\ge #2}}
\newcommand{\qg}[1]{\cQ_G(#1)}

\newcommand{\dgb}[1]{\cD_G^{\mathrm{b}}(#1)}
\newcommand{\dgp}[1]{\cD^+_G(#1)}
\newcommand{\dgm}[1]{\cD^-_G(#1)}

\newcommand{\dgml}[2]{\cD^-_G(#1)^{\le #2}}
\newcommand{\dgpg}[2]{\cD^+_G(#1)^{\ge #2}}
\newcommand{\dgbl}[2]{\cD^{\mathrm{b}}_G(#1)^{\le #2}}
\newcommand{\dgbg}[2]{\cD^{\mathrm{b}}_G(#1)^{\ge #2}}
\newcommand{\dsml}[2]{\cD^-_G(#1)_{\le #2}}

\newcommand{\dspg}[2]{\cD^+_G(#1)_{\ge #2}}

\newcommand{\dsl}[2]{\cD^{\mathrm{b}}_G(#1)_{\le #2}}
\newcommand{\dsg}[2]{\cD^{\mathrm{b}}_G(#1)_{\ge #2}}

\newcommand{\gen}{{\mathrm{gen}}}

\newcommand{\hto}{\hookrightarrow}
\newcommand{\ssm}{\smallsetminus}
\newcommand{\half}{{\textstyle\frac 12}}

\DeclareMathOperator{\im}{im}
\DeclareMathOperator{\alt}{alt}

\DeclareMathOperator{\supp}{supp}
\DeclareMathOperator{\cod}{cod}
\DeclareMathOperator{\scod}{scod}

\DeclareMathOperator{\step}{step}
\DeclareMathOperator{\Hom}{Hom}
\DeclareMathOperator{\End}{End}
\DeclareMathOperator{\Ext}{Ext}

\DeclareMathOperator{\cHom}{\mathcal{H}\mathit{om}}
\DeclareMathOperator{\cRHom}{\mathit{R}\mathcal{H}\mathit{om}}

\newcommand{\Lotimes}{\mathchoice%
  {\overset{\scriptscriptstyle L}{\otimes}}%
  {\otimes^{\scriptscriptstyle L}}{\otimes^L}{\otimes^L}}

\newtheorem{thm}{Theorem}[section]
\newtheorem{lem}[thm]{Lemma}
\newtheorem{prop}[thm]{Proposition}
\newtheorem{cor}[thm]{Corollary}

\theoremstyle{definition}
\newtheorem{defn}[thm]{Definition}

\theoremstyle{remark}
\newtheorem{rmk}[thm]{Remark}

\numberwithin{equation}{section}

\title{Baric structures on triangulated categories and coherent sheaves}

\author{Pramod N. Achar}
\author{David Treumann}
\thanks{The first author was partially supported by NSF Grant DMS-0500873.}
\date{August 23, 2008}

\usepackage{verbatim,color}

\newcommand{\typicalbaric}{(\{\fD_{\leq w}\},\{\fD_{\geq w}\})_{w \in \Z}}
\newcommand{\schemebaric}[1]{(\{\dsl {#1}w\},\{\dsg {#1}w\})_{w \in \Z}}

\newcommand{\dbmc}[1]{\mathrm{D}^{\mathrm{b}}_{\mathrm{m}}(#1)}

\begin{document}

\begin{abstract}
We introduce the notion of a \emph{baric structure} on a triangulated category, as an abstraction of S. Morel's weight truncation formalism for mixed $\ell$-adic sheaves.  We study these structures on the derived category $\dgb X$ of $G$-equivariant coherent sheaves on a $G$-scheme $X$.  Our main result shows how to endow this derived category with a family of nontrivial baric structures when $G$ acts on $X$ with finitely many orbits.

We also describe a general construction for producing a new $t$-structure on a triangulated category equipped with given $t$- and baric structures, and we prove that the staggered $t$-structures on $\dgb X$ introduced by the first author arise in this way.
\end{abstract}

\maketitle

\section{Introduction}
\label{sect:intro}

Let $Z$ be a variety over a finite field.  The triangulated category of
$\ell$-adic sheaves on $X$ has a full subcategory $\dbmc Z$ of ``mixed
sheaves,'' defined in terms of eigenvalues of the Frobenius morphism.  The
existence and good formal properties of this category are among the most
important consequences of Deligne's proof of the Weil conjectures.  It
plays a major role in the theory of perverse sheaves and their applications
in representation theory.  An important part of the formalism of mixed
sheaves is a certain filtration of $\dbmc Z$ by full subcategories 
$\{\dbmc Z_{\le w}\}_{w \in \Z}$, known as the \emph{weight filtration}.

Let us now turn our attention to the world of equivariant coherent sheaves.
 Let $X$ be a scheme (say, of finite type over a field), 
and let $G$ be an affine group scheme acting on $X$ with finitely many orbits.  In~\cite{a}, the first author introduced a class of $t$-structures, called \emph{staggered $t$-structures}, on the bounded derived category $\dgb X$ of $G$-equivariant coherent sheaves on $X$.  These $t$-structures depend on the choice of a certain kind of filtration of the abelian category of equivariant coherent sheaves.  These filtrations, known as \emph{$s$-structures}, bear an at least superficial resemblance to the weight filtration of $\dbmc Z$.

The main goal of this paper is to try to make this resemblance into a precise statement, and to thereby place these two kinds of structures in a unified setting.  We do this by introducing the notion of a \emph{baric structure} on a triangulated category.  The usual weight filtration on $\dbmc Z$ is not a baric structure, but a modified version of it due to S.~Morel~\cite{mor} is.  (Indeed, the definition of a baric structure is largely motivated by Morel's results.)  An $s$-structure is not a baric structure either: for one thing, it is a filtration of an abelian category, not of a triangulated category.  We show in this paper how to construct baric structures on $\dgb X$ using an $s$-structure on $X$.  We also exhibit several other examples of baric structures that have appeared in the literature.

The second goal of the paper is to recast the construction in~\cite{a} as an instance of an abstract operation that can be done on any triangulated category.  Specifically, given a triangulated category with ``compatible'' $t$- and baric structures, we outline a procedure, which we call \emph{staggering}, for producing a new $t$-structure.  Note that in~\cite{a}, ``staggered'' was simply a name assigned to certain specific $t$-structures by definition, whereas in this paper, ``to stagger'' is a verb.  We prove that these two uses of the word are consistent: that is, that the $t$-structures of~\cite{a} arise by staggering the standard $t$-structure on $\dgb X$ with respect to a suitable baric structure.
(The staggering operation can also be applied to the weight baric structure
on $\dbmc Z$, as well as to other baric structures.  This yields a new
$t$-structure that has not previously been studied.)

An outline of the paper is as follows.  We begin in
Section~\ref{sect:baric} by giving the definition of a baric structure and
of the staggering operation. In Section~\ref{sect:examples}, we give
examples of baric structures, including Morel's version of the weight
filtration.  Next, in Section~\ref{sect:baric-coh1}, we begin the study of
baric structures on derived categories of equivariant coherent sheaves,
especially those that behave well with respect to the geometry of the
underlying scheme.

The next three sections are devoted to the relationship between baric
structures and $s$-structures.  First, in Section~\ref{sect:stagt}, we
review relevant definitions and results from~\cite{a}. 
Section~\ref{sect:baric-coh2} contains the main result of the
paper, showing how $s$-structures on the abelian category of coherent
sheaves give rise to baric structures on the derived category.
In Section~\ref{sect:mult}, we briefly consider the reverse problem, that
of producing $s$-structures from baric structures.

Finally, in Section~\ref{sect:stag2}, we study staggered $t$-structures
associated to the baric structures produced in
Section~\ref{sect:baric-coh2}.  Specifically, we prove that their hearts
are finite-length categories, and we give a description of their simple
objects.  This was done in some cases in~\cite{a}, but remarkably, the
machinery of baric structures allows us to remove the assumptions that were
imposed in~{\it loc.~cit}.

We conclude by mentioning an application of the machinery
developed in this paper.  The language of baric structures allows one to
define a notion of ``purity,'' similar to the one for $\ell$-adic mixed
constructible sheaves.  In a subsequent paper~\cite{at}, the authors
prove that every simple staggered sheaf is pure, and that every pure
object in the derived category is a direct sum of shifts of simple
staggered sheaves.  These results are analogous to the well-known Purity
and Decomposition Theorems for $\ell$-adic mixed perverse sheaves.

\section{Baric structures}
\label{sect:baric}

In this section we introduce baric structures on triangulated categories
(Definition~\ref{defn:baric}), and the 
operation of \emph{staggering} a $t$-structure with respect to a baric
structure (Definition~\ref{defn:stag}).
Staggering produces, out of a $t$-structure $(\fD^{\leq 0}, \fD^{\geq 0})$ on a 
triangulated category $\fD$, a new pair of orthogonal subcategories 
$(\sD^{\leq 0},\sD^{\geq 0})$.  Our main result is a criterion which guarantees that
$(\sD^{\leq 0},\sD^{\geq 0})$ is itself a $t$-structure
(Theorem~\ref{thm:stag-gen}).
 
\subsection{Baric structures}

\begin{defn}\label{defn:baric}
Let $\fD$ be a triangulated category.
A \emph{baric structure} on $\fD$ is a pair of collections of thick subcategories $(\{\fD_{\le w}\}, \{\fD_{\ge w}\})_{w \in \Z}$ satisfying the following axioms:
\begin{enumerate}
\item $\fD_{\le w} \subset \fD_{\le w+1}$ and $\fD_{\ge w} \supset \fD_{\ge w+1}$ for all $w$.
\item $\Hom(A,B) = 0$ whenever $A \in \fD_{\le w}$ and $B \in \fD_{\ge w+1}$.
\item For any object $X \in D$, there is a distinguished triangle $A \to X
\to B \to$ with $A \in \fD_{\le w}$ and $B \in \fD_{\ge w+1}$.\label{it:dt}
\end{enumerate}
\end{defn}

This definition is at least superficially very similar to that of \emph{$t$-structure}, 
and in fact arguments identical to those given in~\cite[\S\S 1.3.3--1.3.5]{bbd} yield 
the following basic properties of baric structures.

\begin{prop}\label{prop:baric-basic}
Let $\fD$ be a triangulated category equipped with a baric structure $\typicalbaric$.
The inclusion $\fD_{\le w} \hto \fD$ admits a right adjoint $\beta_{\le w}: \fD \to \fD_{\le w}$, and the inclusion $\fD_{\ge w} \to \fD$ admits a left adjoint $\beta_{\ge w}: \fD \to \fD_{\ge w}$.  There is a distinguished triangle
\[
\beta_{\le w}X \to X \to \beta_{\ge w+1}X \to,
\]
and any distinguished triangle as in Axiom~(3) above is canonically isomorphic to this one.  Furthermore, 
if $v \le w$, then we have the following isomorphisms of functors:
\begin{align*}
\beta_{\le v} \circ \beta_{\le w} &\cong \beta_{\le v} &
\beta_{\ge v} \circ \beta_{\le w} &\cong \beta_{\le w} \circ \beta_{\ge v}
\\
\beta_{\ge w} \circ \beta_{\ge v} &\cong \beta_{\ge w} &
\beta_{\le v} \circ \beta_{\ge w} &\cong \beta_{\ge w} \circ \beta_{\le
v}=0
\qedhere
\end{align*}
\end{prop}

Note that in a baric structure, unlike in a $t$-structure, the 
subcategories $\fD_{\le w}$ and $\fD_{\ge w}$ are required to be 
stable under shifts in both directions, and it is not assumed that there is
an autoequivalence $\fD \to \fD$ taking $\fD_{\le w}$ to, say, $\fD_{\le
w+1}$.  
Moreover, baric truncation functors enjoy the following important property.

\begin{prop}
The baric truncation functors $\beta_{\le w}$ and $\beta_{\ge w}$ take 
distinguished triangles to distinguished triangles.
\end{prop}
\begin{proof}
Let $X \to Y \to Z \to$ be a distinguished triangle in $\fD$, and consider
the natural morphism $\beta_{\le w}X \to X$.  The composition of this
morphism with $X \to Y$ factors through $\beta_{\le w}Y \to Y$ (since
$\Hom(\beta_{\le w}X, Y) \cong \Hom(\beta_{\le w}X, \beta_{\le w}Y)$), so
we obtain a commutative diagram
\[
\xymatrix@=10pt{
\beta_{\le w}X \ar[r]\ar[d] & \beta_{\le w}Y \ar[d] \\
X \ar[r] & Y}
\]
Let us complete this diagram using the
$9$-lemma~\cite[Proposition~1.1.11]{bbd}:
\[
\xymatrix@=10pt{
\beta_{\le w}X \ar[r]\ar[d] & \beta_{\le w}Y \ar[r]\ar[d] & Z' \ar[r]\ar[d]
& \\
X \ar[r]\ar[d] & Y\ar[r]\ar[d] & Z \ar[r]\ar[d] & \\
\beta_{\ge w+1}X \ar[r]\ar[d] & \beta_{\ge w+1}Y \ar[r]\ar[d] & Z''
\ar[r]\ar[d] & \\
&&&}
\]
Since $\fD_{\le w}$ and $\fD_{\ge w+1}$ are full triangulated subcategories
of $\fD$, we see that $Z' \in \fD_{\le w}$ and $Z'' \in \fD_{\ge w+1}$. 
But then Proposition~\ref{prop:baric-basic} tells us that $Z' \cong
\beta_{\le w}Z$ and $Z'' \cong \beta_{\ge w}Z$, so we obtain distinguished
triangles
\[
\beta_{\le w}X \to \beta_{\le w}Y \to \beta_{\le w}Z \to
\qquad\text{and}\qquad
\beta_{\ge w+1}X \to \beta_{\ge w+1}Y \to \beta_{\ge w+1}Z \to,
\]
as desired.
\end{proof}

\begin{defn}
Let $\fD$ be a triangulated category equipped with a baric structure
$\typicalbaric$.
We will use the following terminology:

\begin{enumerate}
\item The adjoints $\beta_{\le w}$ and $\beta_{\ge w}$ to the inclusions 
$\fD_{\le w} \hto \fD$ and $\fD_{\ge w} \hto \fD$ are called \emph{baric 
truncation functors}.

\item The baric structure is \emph{bounded} if for each object $A \in
\fD$, 
there exist integers $v, w$ such that $A \in \fD_{\ge v} \cap \fD_{\le w}$.

\item It is \emph{nondegenerate} if there is no nonzero object belonging to
all 
$\fD_{\le w}$ or to all $\fD_{\ge w}$.  Note that a bounded
baric structure is automatically nondegenerate.

\item Let $\fD'$ be another triangulated category, and suppose it is 
equipped with a baric structure $(\{\fD'_{\le w}\}, \{\fD'_{\ge w}\})$.  
A functor of triangulated categories $F: \fD \to \fD'$ is said to be 
\emph{left baryexact} if $F(\fD_{\ge w}) \subset \fD'_{\ge w}$ for all 
$w \in \Z$, and \emph{right baryexact} if $F(\fD_{\le w}) \subset 
\fD'_{\le w}$ for all $w \in \Z$.
\end{enumerate}
\end{defn}

Let us also record the following definitions, though we will not use them
until later in the paper.

\begin{defn}
Let $\fD$ be a triangulated category equipped with a baric structure
$\typicalbaric$.

\begin{enumerate}
\item Suppose $\fD$ is equipped with an involutive antiequivalence $\D: \fD
\to \fD$.  
The baric structure is \emph{self-dual} if $\D(\fD_{\le w}) = \fD_{\ge
-w}$.

\item Suppose $\fD$ has the structure of a tensor category, with tensor
product $\otimes$.  
The baric structure is \emph{multiplicative} with respect to $\otimes$ if 
if for any $A \in \fD_{\le v}$ and $B \in \fD_{\le w}$, we have $A \otimes
B \in \fD_{\le v+w}$.

\item Suppose $\fD$ has an internal Hom functor $\cHom$.  The baric
structure is \emph{multiplicative}
with respect to $\cHom$ if for any $A \in \fD_{\leq v}$ and $B \in 
\fD_{\geq w}$, we have 
$\cHom(A,B) \in \fD_{\geq w -v}$.
\end{enumerate}
Note that whenever we have an adjunction between $\otimes$ and $\cHom$, the
multiplicativity conditions
are equivalent. 
\end{defn}

\subsection{Staggering}

Below, if $\fD$ is equipped with a $t$-structure $(\fD^{\le 0}, \fD^{\ge
0})$, 
we write $\fC = \fD^{\le 0} \cap \fD^{\ge 0}$ for its heart, and we denote
the 
associated truncation functors by $\tau^{\le n}$ and $\tau^{\ge n}$.  The
$n$th 
cohomology functor associated to the $t$-structure is denoted $h^n: \fD \to
\fC$.

\begin{defn}\label{defn:compat}
Let $\fD$ be a triangulated category equipped with both a $t$-structure 
and a baric structure.  These structures are said to be \emph{compatible} 
if $\tau^{\le n}$ and $\tau^{\ge n}$ are right baryexact, and $\beta_{\le
w}$ 
and $\beta_{\ge w}$ are left $t$-exact.
\end{defn}

\begin{rmk}
Of course there is a dual notion of compatibility, but it does not seem to
arise
as often.
\end{rmk}

\begin{defn}\label{defn:stag}
Let $\fD$ be a triangulated category equipped with compatible $t$- and
baric structures.  Define two full subcategories of $\fD$ as follows:
\begin{align*}
\sD^{\le 0} &= \{A \in D \mid \text{$h^k(A) \in \fD_{\le -k}$ for all $k
\in \Z$} \}, \\
\sD^{\ge 0} &= \{B \in D \mid \text{$\beta_{\le k}B \in \fD^{\ge -k}$ for
all $k \in \Z$} \}.
\end{align*}
Assume that the pair $(\sD^{\le 0}, \sD^{\ge 0})$ constitutes a
$t$-structure. It is called the \emph{staggered $t$-structure}, or the
$t$-structure obtained by \emph{staggering} the original $t$-structure with
respect to the given baric structure.
\end{defn}

As usual, we let $\sD^{\le n} = \sD^{\le 0}[-n]$ and $\sD^{\ge n} =
\sD^{\ge 0}[-n]$.

\begin{lem}\label{lem:compat}
Let $\fD$ be a triangulated category equipped with compatible $t$- and
baric structures.  Assume the $t$-structure is nondegenerate.
\begin{enumerate}
\item $A \in \fD_{\le w}$ if and only if $h^k(A) \in \fD_{\le w}$ for all
$k$. \label{it:lth}
\item $B \in \fD_{\ge w}$ if and only if $\beta_{\le w-1}\tau^{\le k}B \in
\fD^{\ge k+2}$ for all $k$.\label{it:gth}
\item We have \label{it:gthom}
\[
\fD_{\ge w} \cap \fC = \{ B \in \fC \mid
\text{$\Hom^k(A,B) = 0$ for all $A \in \fD_{\le w} \cap \fC$ and all $k \ge
0$} \}.
\]
\item $\fD_{\le w} \cap \fC$ is a Serre subcategory of $\fC$, and $\fD_{\ge
w} \cap \fC$ is stable under extensions.\label{it:ltserre}
\item $\sD^{\le 0}$ and $\sD^{\ge 0}$ are stable under extensions.
\label{it:ssext}
\item $\fD^{\le k} \cap \fD_{\le w} \subset \sD^{\le k+w}$, and $\fD^{\ge
k} \cap \fD_{\ge w} \subset \sD^{\ge k+w}$. \label{it:sscap}
\end{enumerate}
\end{lem}
\begin{proof}
\eqref{it:lth}~Since $\fD_{\le w}$ is stable under $\tau^{\le k}$ and
$\tau^{\ge k}$, it is clear that $A \in \fD_{\le w}$ implies that $h^k(A)
\in \fD_{\le w}$.  Conversely, suppose $h^k(A) \in \fD_{\le w}$ for all
$k$.  Recall (e.g. \cite[Proposition 4.4.6]{verdier}) that we have a spectral sequence
\begin{equation}\label{eqn:e2-1}
E_2^{ab} = \Hom(h^{-b}(A), B[a]) \quad\Longrightarrow\quad
\Hom(A,B[a+b]).
\end{equation}
Since $\Hom(h^{-b}(A), B[a]) = 0$ for all $B \in \fD_{\ge w+1}$ and all $a,
b \in \Z$, we see that $\Hom(A,B) = 0$ for all $B \in \fD_{\ge w+1}$, and
hence that $A \in \fD_{\le w}$.

\eqref{it:gth}~Consider the distinguished triangle
\[
\beta_{\le w-1}\tau^{\le k}B \to \beta_{\le w-1}B \to \beta_{\le
w-1}\tau^{\ge k+1}B \to.
\]
The last term is always in $\fD^{\ge k+1}$ by the left $t$-exactness of
$\beta_{\le w-1}$.  If $B \in \fD_{\ge w}$, so that $\beta_{\le w-1}B = 0$,
then $\beta_{\le w-1}\tau^{\le k}B \cong (\beta_{\le w-1}\tau^{\ge
k+1}B)[-1] \in \fD^{\ge k+2}$.  Conversely, if the $t$-structure is
nondegenerate, and if $\beta_{\le w-1}\tau^{\le k}B \in \fD^{\ge k+2}$ for
all $k$, the distinguished triangle above shows that $\beta_{\le w-1}B \in
\fD^{\ge k+1}$ for all $k$, and hence that $\beta_{\le w-1}B = 0$, so $B
\in \fD_{\ge w}$, as desired.

\eqref{it:gthom}~If $B \in \fD_{\ge w} \cap \fC$, then clearly
$\Hom(A[-k],B) =  0$ for all $A \in \fD_{\le w-1} \cap \fC$ and all $k \ge
0$, since $A[-k] \in \fD_{\le w-1}$ for all $k$.  Conversely, if
$\Hom(A,B[k]) =  0$ for all $A \in \fD_{\le w-1} \cap \fC$ and all $k \ge
0$, the spectral sequence~\eqref{eqn:e2-1} shows that $\Hom(A,B) = 0$ for
all $A \in \fD_{\le w-1}$, and hence that $B \in \fD_{\ge w}$.

\eqref{it:ltserre}~Suppose we have a short exact sequence
\[
0 \to A \to B \to C \to 0
\]
in $\fC$.  If $A$ and $C$ are in $\fD_{\le w}$, then $B$ must be as well,
since $\fD_{\le w}$ is stable under extensions.  Conversely, suppose $B \in
\fD_{\le w}$.  Assume that $C \notin \fD_{\le w}$, and consider the
distinguished triangle
\[
\beta_{\le w}C \to C \to \beta_{\ge w+1}C \to.
\]
By left $t$-exactness of the baric truncation functors, we have an exact
sequence
\[
0 \to h^0(\beta_{\le w}C) \to C \to h^0(\beta_{\ge w+1}C).
\]
We must have $h^0(\beta_{\ge w+1}C) \ne 0$: otherwise, we would have $C
\cong h^0(\beta_{\le w}C) \in \fD_{\le w}$.  Next, from the distinguished
triangle
\[
\beta_{\ge w+1}A \to 0 \to \beta_{\ge w+1}C \to,
\]
we see that $\beta_{\ge w+1}A \cong \beta_{\ge w+1}C[-1]$.  In particular,
$h^0(\beta_{\ge w+1}A) = 0$.  But then the exact sequence
\[
0 \to h^0(\beta_{\le w}A) \to A \to h^0(\beta_{\ge w+1}A) = 0
\]
shows that $A \cong h^0(\beta_{\le w}A) \in \fD_{\le w}$, and hence that
$\beta_{\ge w+1}A = 0$ and $\beta_{\ge w+1}C = 0$.  Thus, $A$ and $C$ are
in $\fD_{\le w}$, as desired.

That $\fD_{\ge w} \cap \fC$ is stable under extensions follows immediately
from the fact that $\fD_{\ge w}$ is stable under extensions.

\eqref{it:ssext}~Let $A \to B \to C \to$ be a distinguished triangle with
$A \in \sD^{\le 0}$ and $C \in \sD^{\le 0}$, and consider the exact
sequence
\[
h^k(A) \overset{f}{\to} h^k(B) \overset{g}{\to} h^k(C).
\]
Since $h^k(A) \in \fD_{\le -k}$, its quotient $\im f$ is in $\fD_{\le -k}$
as well.  Similarly, $\im g \in \fD_{\le -k}$ because it is a subobject
of $h^k(C)$.  Now, from the short exact sequence $0 \to \im f \to h^k(B)
\to \im g \to 0$, we deduce that $h^k(B) \in \fD_{\le -k}$.  Thus, $B \in
\sD^{\le 0}$.

On the other hand, if $A \to B \to C \to$ is a distinguished triangle with
$A, C \in \sD^{\ge 0}$, consider the distinguished triangle
\[
\beta_{\le k}A \to \beta_{\le k}B \to \beta_{\le k}C \to.
\]
Since $\beta_{\le k}A$ and $\beta_{\le k}C$ lie in $\fD^{\ge -k}$,
$\beta_{\le k}B \in \fD^{\ge -k}$ as well, so $B \in \sD^{\ge 0}$.

\eqref{it:sscap}~If $A \in \fD^{\le k} \cap \fD_{\le w}$, then $h^i(A[k+w])
= h^{i+k+w}(A) = 0$ if $i > -w$, and $h^i(A[k+w]) \in \fD_{\le w} \subset
\fD_{\le -i}$ if $i \le -w$.  Thus, $A[k+w] \in \sD^{\le 0}$, or $A \in
\sD^{\le k+w}$.  Next, suppose $B \in \fD^{\ge k} \cap \fD_{\ge w}$.  Then
$\beta_{\le i}B[k+w] = 0$ if $i < w$, and $\beta_{\le i}B[k+w] \in \fD^{\ge
k}[k+w] = \fD^{\ge -w} \subset \fD^{\ge -i}$ if $i \ge w$.  Hence, $B[k+w]
\in \sD^{\ge 0}$, or $B \in \sD^{\ge k+w}$.
\end{proof}

\begin{prop}\label{prop:compat}
Let $\fD$ be a triangulated category equipped with compatible $t$- and
baric structures.  Assume the $t$-structure is nondegenerate.
\begin{enumerate}
\item $\Hom(A,B) = 0$ for all $A \in \sD^{\le 0}$ and $B \in \sD^{\ge 1}$.
\label{it:sshom}
\item If $\Hom(A,B) = 0$ for all $B \in \sD^{\ge 1}$, then $A \in \sD^{\le
0}$.  If $\Hom(A,B) = 0$ for all $A \in \sD^{\le 0}$, then $B \in \sD^{\ge
1}$. \label{it:ssorth}
\item $\sD^{\le 0} \subset \sD^{\le 1}$ and $\sD^{\ge 0} \supset \sD^{\ge
1}$. \label{it:ssshift}
\item If the baric structures is also nondegenerate, there is no nonzero
object belonging to all $\sD^{\le n}$ or to all $\sD^{\ge n}$.
\label{it:ssnondeg}
\item If the $t$- and baric structures are bounded, then for any $A \in D$,
there are integers $n, m$ such that $A \in \sD^{\ge n} \cap \sD^{\le m}$.
\label{it:ssbdd}
\end{enumerate}
\end{prop}
\begin{proof}
\eqref{it:sshom}~For any $k \in \Z$, $h^{-k}(A) \in \fD_{\le k}$, and
therefore $\Hom(h^{-k}(A),B[k]) \cong \Hom(h^{-k}(A), \beta_{\le k}B[k])$. 
But $\beta_{\le k}B \in \fD^{\ge k+1}$, so $\Hom(h^{-k}(A), \beta_{\le
k}B[k]) = 0$ for all $k$.  It follows from the spectral
sequence~\eqref{eqn:e2-1} that $\Hom(A,B) = 0$.

\eqref{it:ssorth}~Suppose $\Hom(A,B) = 0$ for all $B \in \sD^{\ge 1}$, and
suppose for some $k$, $h^k(A) \notin \fD_{\le -k}$.  That implies that
$\tau^{\ge k}A \notin \fD_{\le -k}$, so $\beta_{\ge -k+1}\tau^{\ge k}A \ne
0$.  In particular, the natural adjunction morphism $A \to \beta_{\ge
-k+1}\tau^{\ge k}A$ is nonzero.  However, $\beta_{\ge -k+1}\tau^{\ge k}A
\in \fD^{\ge k} \cap \fD_{\ge -k+1} \subset \sD^{\ge 1}$.  This contradicts
the assumption that $\Hom(A,B) = 0$ for all $B \in \sD^{\ge 1}$, so we must
have $h^k(A) \in \fD_{\le -k}$ for all $k$, and hence $A \in \sD^{\le 0}$.

On the other hand, if $\Hom(A,B) = 0$ for all $A \in \sD^{\le 0}$, a
similar argument involving the morphism $\tau^{\le -k}\beta_{\le k}B \to B$
shows that $B \in \sD^{\ge 1}$.

\eqref{it:ssshift}~If $A \in \sD^{\le 0}$, then $h^k(A[1]) = h^{k+1}(A) \in
\fD_{\le -k-1} \subset \fD_{\le -k}$, so $A[1] \in \sD^{\le 0}$, and hence
$\sD^{\le 0} \subset \sD^{\le 1}$.  Similarly, if $B \in \sD^{\ge 0}$, then
$\beta_{\le k}B[-1] \in \fD^{\ge -k+1} \subset \fD^{\ge -k}$, so $B[-1] \in
\sD^{\ge 0}$.

\eqref{it:ssnondeg}~Suppose $A \in \sD^{\le n}$ for all $n$.  Then $h^k(A)
\in \fD_{\le n-k}$ for all $n$ and all $k$.  The nondegeneracy of the baric
structure implies that $h^k(A) = 0$; then, the nondegeneracy of the
$t$-structure implies that $A = 0$.  Next, suppose $A \in \sD^{\ge n}$ for
all $n$, and assume $A \ne 0$.  Choose some $w$ such that $\beta_{\le w}A
\ne 0$, and then choose some $k$ such that $\tau^{\le k}\beta_{\le w}A \ne
0$.  By right baryexactness of $\tau^{\le k}$, we know that $\tau^{\le
k}\beta_{\le w}A \in \fD_{\le w}$, so we obtain a sequence of isomorphisms
\[
\Hom(\tau^{\le k}\beta_{\le w}A, \tau^{\le k}\beta_{\le w}A)
\cong \Hom(\tau^{\le k}\beta_{\le w}A, \beta_{\le w}A)
\cong \Hom(\tau^{\le k}\beta_{\le w}A, A).
\]
In particular, the natural map $\tau^{\le k}\beta_{\le w}A \to A$ is
nonzero.  But clearly $\tau^{\le k}\beta_{\le w}A \in \sD^{\le k+w}$, so $A
\notin \sD^{\ge k+w+1}$, a contradiction.

\eqref{it:ssbdd}~This follows from Lemma~\ref{lem:compat}\eqref{it:sscap}.
\end{proof}

We will not prove in general that $(\sD^{\le 0}, \sD^{\ge 0})$ is a
$t$-structure.

\begin{thm}\label{thm:stag-gen}
Let $\fD$ be a triangulated category endowed with compatible bounded,
nondegenerate $t$- and baric structures.  Suppose we have a function $\mu:
\fD \to \N$ with the following properties:
\begin{enumerate}
\item $\mu(X) = 0$ if and only if $X = 0$.
\item If $X \in \fD^{\ge n}$ but $X \notin \fD^{\ge n+1}$, then
$\mu(\tau^{\ge n+1}\beta_{\le -n}X) < \mu(X)$.
\end{enumerate}
Then $(\sD^{\le 0}, \sD^{\ge 0})$ is a bounded, nondegenerate $t$-structure
on $\fD$.
\end{thm}
\begin{proof}
It will be convenient to use ``$*$'' operation on triangulated categories
({\it cf.} \cite[\S 1.3.9]{bbd}): given two classes of objects
$\cA, \cB \subset \fD$, we denote by $\cA * \cB$ the class of all objects
$X \in \fD$ such that there exists a distinguished triangle $A \to X \to B \to$ with $A \in \cA$ and $B \in \cB$.  In view of the preceding proposition, the present theorem will be proved once we show that every object of $\fD$ belongs to $\sD^{\le 0} * \sD^{\ge 1}$.  We proceed by
induction on $\mu(X)$.  If $\mu(X) = 0$, then $X = 0$, and there is nothing
to prove.  Otherwise, let $n$ be the smallest integer such that $h^n(X) \ne
0$.  Let $A_1 = \tau^{\le n}\beta_{\le -n} X$, $X' = \tau^{\ge
n+1}\beta_{\le -n} X$, and $B_1 = \beta_{\ge -n+1} X$.  It follows from
the right baryexactness of $\tau^{\le n}$ that $A_1 \in \sD^{\le 0}$, and,
similarly, it follows from the left $t$-exactness of $\beta_{\ge -n+1}$
that $B_1 \in \sD^{\ge 1}$.  Recall~\cite[Proposition~1.3.10]{bbd} that the ``$*$'' operation is associative.  By construction, we have
\[
X \in \{A_1\} * \{X'\} * \{B_1\} \subset \sD^{\le 0} * \{X'\} * \sD^{\ge 1}.
\]
Since $\mu(X') < \mu(X)$ by assumption, we know that $X' \in \sD^{\le 0} *
\sD^{\ge 1}$, and hence
\[
X \in \sD^{\le 0} * \sD^{\le 0} * \sD^{\ge 1} * \sD^{\ge 1}.
\]
Since $\sD^{\le 0}$ and $\sD^{\ge 1}$ are stable under extensions, we have
$\sD^{\le 0} * \sD^{\le 0} = \sD^{\le 0}$ and $\sD^{\ge 1} * \sD^{\ge 1} =
\sD^{\ge 1}$, so $X \in \sD^{\le 0} * \sD^{\ge 1}$, as desired.
\end{proof}

\section{Examples}
\label{sect:examples}

In this section, we exhibit several examples of baric structures occurring
``in nature.''  In the first one, the staggering operation of
Definition~\ref{defn:stag} is a new approach to a known $t$-structure.  In
two others, this operation gives what appears to be a previously unknown
$t$-structure.  The main example of this paper---baric structures on
derived categories of coherent sheaves---will be discussed in the next section.

\subsection{Perverse sheaves}

Let $X$ be a topologically stratified space (as in~\cite{gm:ih}), with all
strata of even real dimension.  (This example can be easily modified to
relax that condition, or to treat stratified varieties over a field
instead.) Let $D = D^b_c(X)$ be the bounded derived category of sheaves of
complex vector spaces that are constructible with respect to the given
stratification.  For any $w
\in \Z$, let $X_w$ be the union of all strata of dimension at most $2w$. 
(Thus, $X_w = \varnothing$ if $w < 0$.)  This is a closed subspace of $X$. 
Let $i_w: X_w \to X$ be the inclusion map.  Let $D_{\le w}$ be the full
subcategory consisting of complexes whose support is contained in $X_w$,
and let $D_{\ge w+1}$ be the full subcategory of complexes $\cF$ such that
$i_w^!\cF = 0$.

If $\cF \in D_{\le w}$ and $\cG \in D_{\ge w+1}$, then $\cF \cong
i_{w*}i_w^{-1}\cF$, and 
\[
\Hom(\cF,\cG) \cong \Hom(i_{w*}i_w^{-1}\cF,\cG) \cong \Hom(i_w^{-1}\cF,
i_w^!\cG) = 0.
\]
Next, let $j_{w+1}: (X \ssm X_w) \to X$ be the open inclusion of the
complement of $X_w$.  For any complex $\cF$, the distinguished triangle
\[
i_{w*}i_w^!\cF \to \cF \to (j_{w+1})_*j_{w+1}^{-1}\cF \to
\]
is one whose first term lies in $D_{\le w}$ and whose last term lies in
$D_{\ge w+1}$.  Thus, we see that $(\{D_{\le w}\}, \{D_{\ge w}\})_{w \in \Z}$ is a
baric structure on $D^b_c(X)$, with baric truncation functors
\[
\beta_{\le w} = i_{w*}i_w^!
\qquad\text{and}\qquad
\beta_{\ge w} = j_{w*}j_w^{-1}.
\]

It is easy to see that this baric structure is compatible with the standard
$t$-structure on $D$.  If $\cF$ is supported on $X_w$, it is obvious that
any truncation of it is as well, so $D_{\le w}$ is stable under $\tau^{\le
n}$ and $\tau^{\ge n}$.  On the other hand, it is clear from the formulas
above that $\beta_{\le w}$ and $\beta_{\ge w}$ are both left $t$-exact.

In the associated staggered $t$-structure $(\s D^{\le 0}, \s D^{\ge 0})$,
we have $\cF \in \s D^{\le 0}$ if and only if $h^k(\cF) \in D_{\le -k}$,
or, in other words,
\[
\dim \supp h^k(\cF) \le -2k.
\]
The staggered $t$-structure in this case is none other than the perverse
$t$-structure of middle perversity.

\subsection{Quasi-exceptional sets}

Let $\fD$ be a triangulated category.  A set of objects $\{\nabla^w\}_{w
\in \N}$ in $\fD$ indexed by nonnegative integers is called a
\emph{quasi-exceptional set} if the following conditions hold:
\begin{enumerate}
\item If $v < w$, then $\Hom(\nabla^v, \nabla^w[k]) =0$ for all $k \in
\Z$.\label{it:exc-hom}
\item For any $w \in \N$, $\Hom(\nabla^w, \nabla^w[k]) = 0$ if $k < 0$, and
$\End(\nabla^w)$ is a division ring.\label{it:exc:end}
\end{enumerate}
For $w \in \N$, let $\fD_{\le w}$ be the full triangulated subcategory of
$\fD$ generated by $\nabla^0, \ldots, \nabla^w$, and for an integer $w <
0$, let $\fD_{\le w}$ be the full triangulated subcategory containing only
zero objects. (Here, we are following the notation of~\cite{bez:qes}, but
this will turn out to be consistent with our notation for baric structures
as well.)  A quasi-exceptional set is \emph{dualizable} if there is another
collection of objects $\{\Delta_w\}_{w \in \N}$ such that
\begin{enumerate}
\setcounter{enumi}{2}
\item If $v > w$, $\Hom(\Delta_v, \nabla^w[k]) = 0$ for all $k \in
\Z$.\label{it:dexc-hom}
\item For any $w \in \N$, we have $\Delta_w \cong \nabla^w \mod
\fD_{\le w-1}$.\label{it:dexc-iso}
\end{enumerate}
The last condition means that $\Delta_w$ and $\nabla^w$ give rise to
isomorphic objects in the quotient category $\fD_{\le w}/\fD_{\le w-1}$.

Next, let $\fD_{\ge w}$ be the full triangulated subcategory generated by
the objects $\{\nabla^k \mid k \ge w\}$.  If $A \in \fD_{\le w}$ and $B
\in \fD_{\ge w+1}$, then Axiom~\eqref{it:exc-hom} above implies that
$\Hom(A,B) = 0$.  In addition, by~\cite[Lemma~4(e)]{bez:qes}, each
inclusion $\fD_{\le w} \to \fD_{\le w+1}$ admits a right adjoint $\iota_w$.
By a straightforward argument, these functors can be used to construct
distinguished triangles as in Definition~\ref{defn:baric}\eqref{it:dt}. 
Thus, $(\{\fD_{\le w}\}, \{\fD_{\ge w}\})_{w \in \Z}$ is a baric structure
on $\fD$.  It is nondegenerate and bounded by construction.

A key result of~\cite{bez:qes} is the construction of a
bounded, nondegenerate $t$-structure $(\fD^{\le 0}, \fD^{\ge 0})$
associated to a quasi-exceptional set.  This $t$-structure is defined as
follows (see~\cite[Proposition~1]{bez:qes}):
\begin{align*}
\fD^{\le 0} &= \langle \{ \Delta_w[n] \mid n \ge 0 \} \rangle, \\
\fD^{\ge 0} &= \langle \{ \nabla_w[n] \mid n \le 0 \} \rangle.
\end{align*}
Here, the notation $\langle S \rangle$ stands for the smallest strictly
full subcategory of $\fD$ that is stable under extensions and contains all
objects in the set $S$.  

We claim that this $t$-structure and the baric
structure defined above are compatible.  It follows from
Axiom~\eqref{it:exc-hom} above that
\[
\beta_{\le w}\nabla^v =
\begin{cases}
0 & \text{if $w < v$,} \\
\nabla^v & \text{if $w \ge v$,}
\end{cases}
\qquad\text{and}\qquad
\beta_{\ge w}\nabla^v =
\begin{cases}
0 & \text{if $w > v$,} \\
\nabla^v & \text{if $w \le v$.}
\end{cases}
\]
This calculation shows that the baric truncation functors preserve
$\fD^{\ge 0}$.  On the other hand, Axiom~\eqref{it:dexc-hom} implies that
$\tau^{\le 0}\nabla^w$ is contained in the subcategory generated by
$\Delta_0, \ldots, \Delta_w$, and that subcategory coincides with
$\fD_{\le w}$ by Axiom~\eqref{it:dexc-iso}.  Thus, $\tau^{\le 0}$
preserves $\fD_{\le w}$, so $\tau^{\ge 0}$ does as well.

Finally, given a nonzero object $X \in \fD$, let $a(X)$ be the smallest
integer $n$ such that $X \in \fD^{\ge -n}$, and let $b(X)$ be the smallest
integer $w$ such that $X \in \fD_{\le w}$.  Note that $b(X) \ge 0$.  Let
\[
\mu(X) =
\begin{cases}
\max \{a(X)+1,b(X)\} + 1 & \text{if $X \ne 0$,} \\
0 & \text{if $X = 0$.}
\end{cases}
\]
Clearly, $\mu$ takes nonnegative integer values, and $\mu(X) = 0$ if and
only if $X = 0$.  Moreover, if $a(X) = -n$ (which implies $\mu(X) \ge
-n+2$), then $a(\tau^{\ge n+1}\beta_{\le -n}X) \le -n-1$ and $b(\tau^{\ge
n+1}\beta_{\le -n}X) \le -n$, so $\mu(\tau^{\ge n+1}\beta_{\le-n}X) \le
-n+1$.  Thus, the conditions of Theorem~\ref{thm:stag-gen} are satisfied,
and there is a staggered $t$-structure $(\sD^{\le 0}, \sD^{\ge 0})$ on
$\fD$.

\subsection{Weight truncation for $\ell$-adic mixed constructible sheaves}

Let $X$ be a scheme of finite type over a finite field $\F_q$, and let
$\ell$ be a fixed prime number distinct from the characteristic of $\F_q$. 
Let $D = D^b_m(X,\Q_\ell)$ be the bounded derived category of mixed
constructible $\Q_\ell$-sheaves on $X$.  Let $\p h^n$ denote the $n$th
cohomology functor with respect to the perverse $t$-structure on $D$ with
respect to the middle perversity.  Let $D_{\le w}$ (resp.~$D_{\ge w}$) be
the full subcategory of $D^b_m(X,\Q_\ell)$ consisting of objects $\cF$ such
that $\p h^n(\cF)$ is of weight $\le w$ (resp.~$\ge w$) for all $n \in \Z$.
 S.~Morel has shown~\cite[Proposition~4.1.1]{mor} that $(\{D_{\le w}\}, \{D_{\ge
w}\})_{w\in\Z}$ is a baric structure on $D^b_m(X, \Q_\ell)$.

Since all objects in the heart of this $t$-structure have finite length, we
may attach a nonnegative integer $\mu(\cF)$ to each complex $\cF$ by the
formula
\[
\mu(\cF) = \sum_{n \in \Z} (\text{length of $\p h^n(\cF)$}).
\]
Moreover, by~\cite[Proposition~4.1.3]{mor}, the baric truncation functors
are $t$-exact for the perverse $t$-structure.  This implies that $\mu$
satisfies the assumptions of Theorem~\ref{thm:stag-gen}, so the perverse
$t$-structure on $D^b_m(X, \Q_\ell)$ can be staggered with respect to
Morel's baric structure to obtain a new $t$-structure.  The authors are not
aware of any previous appearance of this ``staggered-perverse''
$t$-structure on $\ell$-adic mixed constructible sheaves.

\subsection{Diagonal complexes}

We conclude with an example, due to T.~Ekedahl~\cite{eke}, of a
$t$-structure that closely resembles a staggered $t$-structure, although
it does not in general arise by staggering with respect to a baric
structure.  (The authors thank N.~Ramachandran for pointing out this work
to them.)  Let $\fD$ be a triangulated category with a bounded,
nondegenerate $t$-structure $(\fD^{\le 0}, \fD^{\ge 0})$, and as usual,
let $\fC = \fD^{\le 0} \cap \fD^{\ge 0}$.  Suppose $\{\fC_{\le w}\}_{w \in
\Z}$ is an increasing collection of Serre subcategories of $\fC$, and let
$\fC_{\ge w} = \{ B \in \fC \mid \text{$\Hom(A,B) = 0$ for all $A \in
\fC_{\le w-1}$} \}$.  Following Ekedahl, the collection $\{\fC_{\le w}\}$
is called a \emph{radical filtration} of the pair $(\fD, \fC)$ if the
following axioms hold:
\begin{enumerate}
\item For each object $A \in \fC$, there exist integers $v, w$ such that
$A \in \fC_{\ge v} \cap \fC_{\le w}$.
\item If $A \in \fC_{\le w}$ and $B \in \fC_{\ge v}$, then $\Hom^{v-w-1}(A,
B) = 0$ in $\fD$.
\end{enumerate}
If $(\fD, \fC)$ is equipped with a radical filtration, Ekedahl shows that
the categories
\begin{align*}
\tilde \fD^{\le 0} &= \{ A \in \fD \mid \text{$h^k(A) \in \fC_{\le-k}$ for
all $k \in \Z$} \}, \\
\tilde \fD^{\ge 0} &= \{ B \in \fD \mid \text{$h^k(B) \in \fC_{\ge-k}$ for
all $k \in \Z$} \}
\end{align*}
constitute a bounded, nondegenerate $t$-structure on $\fD$.  This is
called the \emph{diagonal $t$-structure}, and the objects in its heart are
called \emph{diagonal complexes}.

These formulas are, of course, strongly reminiscent of those in
Definition~\ref{defn:stag}.  Let us comment briefly on the relationship
between the two constructions.  Given a radical filtration, one could hope
to define a baric structure by setting $\fD_{\le w} = \{ A \in \fD \mid
h^k(A) \in \fC_{\le w}\text{ for all }k \in \Z \}$.  However, the
construction of a baric truncation functor turns out to require a stronger
Hom-vanishing condition between $\fC_{\le w}$ and $\fC_{\ge w+1}$ than
that stated above: one needs something like
Lemma~\ref{lem:compat}\eqref{it:gthom}.  Conversely, given a
baric structure, one could hope to define a radical filtration by setting
$\fC_{\le w} = \fD_{\le w} \cap \fC$.  This also fails, because a baric
structure imposes no higher Hom-vanishing conditions on the
right-orthogonal of $\fC_{\le w}$.

\section{Baric Structures on Coherent Sheaves, I}
\label{sect:baric-coh1}

In this section, we will investigate baric structures on derived categories
of coherent sheaves.  Let $X$ be a scheme of finite type over a noetherian
base scheme, and let $G$ be an affine group scheme over the same base,
acting on $X$.  We adopt the convention that all statements about
subschemes are to be understood in the $G$-invariant sense.  Thus, ``open
subscheme'' will always mean ``$G$-stable open subscheme,'' and
``irreducible'' will mean ``not a union of two proper $G$-stable closed
subschemes.''  This convention will remain in effect for the remainder of
the paper.

Let $\cg X$ and $\qg X$ denote the categories of
$G$-equivariant coherent and quasicoherent sheaves, respectively, on $X$.
One of the headaches of the subject is the need to work with three closely
related triangulated categories, which we denote as follows:
\begin{enumerate}
\item[(1)] $\dgb X$ is the bounded derived category of $\cg X$.
\item[(2)] $\dgm X$ is the bounded-above derived category of $\cg X$.
\item[(3)] $\dgp X$ is the full subcategory of the bounded-below derived
category of
$\qg X$ consisting of objects with coherent cohomology sheaves.
\end{enumerate}
$\dgb X$ will be the focus of our attention, but it will be necessary to
work $\dgm X$ and $\dgp X$ as well, simply because most operations on
sheaves take values in one of those categories, even when acting on bounded
complexes.

\begin{defn}
A \emph{baric structure} on $X$ is a baric structure on $\dgb X$ which is
compatible 
with the standard $t$-structure.
\end{defn}

\begin{rmk}
\label{rmks:schemebaric}
Implicit in this definition are some finiteness conditions; {\it e.g.}, it is
conceivable that there are interesting baric structures on $\dgp X$ that
take advantage of the fact that the functors $\beta_{\leq w}$ can take
bounded complexes to unbounded complexes.  Nevertheless, this is the
definition we will work with.
\end{rmk}

Inspired by parts~\eqref{it:lth} and~\eqref{it:gth} of
Lemma~\ref{lem:compat}, we define the following subcategories of $\dgm X$
and $\dgp X$:
\begin{align*}
\dsml Xw &= \{ \cF \in \dgm X \mid \text{$h^k(\cF) \in \dsl Xw$ for all
$k$} \}, \\
\dspg Xw &= \{ \cF \in \dgp X \mid \text{$\beta_{\le w-1}\tau^{\le k}\cF
\in \dgbg X{k+2}$ for all $k$} \}.
\end{align*}
It is unknown whether these categories constitute parts of baric structures
on $\dgm X$ or on $\dgp X$.  Nevertheless, they will be useful
in the sequel, in part because they admit the alternate characterization
given in the lemma below.  If $Y$ is another scheme endowed with a baric
structure, we will, by a minor abuse of terminology, call a functor $\dgm X
\to \dgm Y$ \emph{right baryexact} if it takes objects of $\dsml Xw$ to
objects of $\dsml Yw$.  Similarly, we call a functor $\dgp X \to \dgp Y$
\emph{left baryexact} if it takes objects of $\dspg Xw$ to $\dspg Yw$.

\begin{lem}\label{lem:unbdd-orth}
\begin{enumerate}
\item For $\cF \in \dgm X$, we have $\cF \in \dsml Xw$ if and only if
$\Hom(\cF,\cG) = 0$ for all $\cG \in \dsg X{w+1}$.
\item For $\cF \in \dgp X$, we have $\cF \in \dspg Xw$ if and only if
$\Hom(\cG,\cF) = 0$ for all $\cG \in \dsl X{w-1}$.
\end{enumerate}
\end{lem}
In particular, we see from this lemma that
\begin{equation}\label{eqn:pm-bdd}
\begin{aligned}
\dsml Xw \cap \dgb X &= \dsl Xw, \\
\dspg Xw \cap \dgb X &= \dsg Xw.
\end{aligned}
\end{equation}
\begin{proof}
(1)~Suppose $\cF \in \dsml Xw$.  By Lemma~\ref{lem:compat}\eqref{it:lth},
$\tau^{\ge k}\cF \in \dsl Xw$ for all $k$.  In particular, given $\cG \in
\dsg X{w+1}$, let $k$ be such that $\cG \in \dgbg Xk$.  Then $\Hom(\cF,\cG)
\cong \Hom(\tau^{\ge k}\cF, \cG) = 0$.  Conversely, suppose $\cF \in \dgm
X$ but $\cF \notin \dsml Xw$, so that for some $k$, $h^k(\cF) \notin \dsl
Xw$.  Then $\tau^{\ge k}\cF \notin \dsl Xw$.  Let $\cG = \beta_{\ge
w+1}\tau^{\ge k}\cF$.  We then have a nonzero morphism $\tau^{\ge k}\cF \to
\cG$.  Moreover, since the baric structure on $\dgb X$ is compatible with
the standard $t$-structure, we have that $\cG \in \dgbg Xk$, so there is a
natural isomorphism $\Hom(\tau^{\ge k}\cF,\cG) \cong \Hom(\cF,\cG)$.  Thus,
$\Hom(\cF,\cG) \ne 0$.

(2)~Suppose $\cF \in \dgpg Xw$.  Given $\cG \in \dsl X{w-1}$, let $k$ be
such that $\cG \in \dgbl Xk$.  Then $\Hom(\cG,\cF) \cong \Hom(\cG,\tau^{\le
k}\cF) \cong \Hom(\cG,\beta_{\le w-1}\tau^{\le k}\cF) = 0$.  Conversely, if
$\cF \in \dgp X$ but $\cF \notin \dgpg Xw$, then for some $k$, $\beta_{\le
w-1}\tau^{\le k}\cF \notin \fD^{\ge k+2}$.  Let $\cG = \tau^{\le
k+1}\beta_{\le w-1}\tau^{\le k}\cF$.  Then clearly $\cG \in \dgbl X{k+1}$ and
$\cG \in \dsl X{w-1}$, and there is a nonzero morphism $\cG \to \beta_{\le
w-1}\tau^{\le k}\cF$.  In particular, the group $\Hom(\cG, \beta_{\le w-1}\tau^{\le k}\cF) \cong \Hom(\cG, \tau^{\le k}\cF)$ is nonzero.  Now, consider the exact sequence
\[
\Hom(\cG, (\tau^{\ge k+1}\cF)[-1]) \to \Hom(\cG, \tau^{\le k}\cF) \to \Hom(\cG, \cF).
\]
The first term vanishes because $(\tau^{\ge k+1}\cF)[-1] \in \dgpg X{k+2}$, so the natural map $\Hom(\cG, \tau^{\le k}\cF) \to \Hom(\cG, \cF)$ is injective.  It follows that $\Hom(\cG,\cF) \ne 0$.
\end{proof}

\subsection{HLR baric structures}
\label{sect:hlr}

We do not wish to work with arbitrary baric structures on $\dgb X$; rather,
we want them to be well-behaved in relation to the scheme structure on
$X$.  We have already imposed the condition that the baric structure be
compatible with the standard $t$-structure.  We may also ask that it
give rise to baric structures on subschemes, in the following sense.

\begin{defn}\label{defn:induced}
Suppose $X$ is equipped with a baric structure, and let $\kappa: Y \hto X$
be a locally closed subscheme.  A baric structure
on $Y$ is said to be \emph{induced} by the one on $X$ if $L\kappa^*$ is
right baryexact and $R\kappa^!$ is left baryexact.
\end{defn}

The class of ``HLR (hereditary, local, and rigid) baric structures,''
defined below, is particularly well-behaved.  For instance, every locally
closed subscheme of a scheme with an HLR baric structure admits a unique
induced baric structure.  (See Theorem~\ref{thm:hlr-induced}.)  The
remainder of Section~\ref{sect:baric-coh1} is devoted to establishing
various properties of HLR baric structures, and the main result of the
paper, Theorem~\ref{thm:main}, is a statement about a class of nontrivial
HLR baric structures.

\begin{defn}\label{defn:hlr}
A baric structure on $X$ is said to be
\emph{hereditary} if every closed subscheme admits an induced
baric structure.  A hereditary baric structure on $X$ is said to be
\emph{local} if every open subscheme admits an induced baric structure that
is also hereditary.

Next, a hereditary baric structure on $X$ is \emph{rigid} if for every
sequence of closed subschemes $Z \overset{t}{\hto} Z_1 \hto X$ where $Z_1$
is a nilpotent thickening of $Z$ ({\it i.e.}, $Z_1$ has the same underlying
topological space as $Z$), the induced baric structures on $Z$ and $Z_1$
are related as follows:
\begin{equation}\label{eqn:rigid}
\begin{aligned}
\dsl {Z_1}w &= \text{the thick closure of $t_*(\dsl Zw)$,} \\
\dsg {Z_1}w &= \text{the thick closure of $t_*(\dsg Zw)$.}
\end{aligned}
\end{equation}

Finally, a baric structure that is hereditary, local, and rigid is called
an \emph{HLR baric structure}.
\end{defn}

It turns out that the ``local'' and ``rigid'' conditions on an HLR baric structure are redundant:

\begin{thm}\label{thm:hlr}
Every hereditary baric structure is HLR.
\end{thm}

This theorem will be proved in Section~\ref{sect:hlrproof}.  We first
require a couple of preliminary lemmas about induced baric structures,
proved below.  Following that, in Section~\ref{sect:hlrprop}, we will
establish a number of useful properties of HLR baric structures.

\begin{lem}\label{lem:induced}
Let $\schemebaric X$ be a baric structure on $X$, and let $i: Z \hto X$ be
a closed subscheme.  If $Z$
admits an induced baric structure, it is given by
\begin{equation}\label{eqn:hered}
\begin{aligned}
\dsl Zw &= \{ \cF \in \dgb Z \mid i_* \cF \in \dsl Xw \}, \\
\dsg Zw &= \{ \cF \in \dgb Z \mid i_* \cF \in \dsg Xw \}.
\end{aligned}
\end{equation}
Conversely, if the categories~\eqref{eqn:hered} constitute a baric
structure on $Z$, then that baric structure is induced from the one on $X$.

If an open subscheme $j: U \hto X$ admits an induced baric structure, it
is given by
\begin{equation}\label{eqn:local}
\begin{aligned}
\dsl Uw &= \{ \cF \in \dgb U \mid \text{$\cF \cong j^*\cF_1$ for some
$\cF_1 \in \dsl Xw$} \}, \\
\dsg Uw &= \{ \cF \in \dgb U \mid \text{$\cF \cong j^*\cF_1$ for some
$\cF_1 \in \dsg Xw$} \}.
\end{aligned}
\end{equation}
Conversely, if the categories~\eqref{eqn:local} constitute a baric
structure on $U$, then that baric structure is induced from the one on $X$.
\end{lem}
\begin{proof}
Let $\schemebaric Z$ be an induced baric structure on a closed subscheme
$i: Z \hto X$.  If $\cF \in \dsl Zw$, then for all $\cG \in \dspg X{w+1}$,
we have (by Lemma~\ref{lem:unbdd-orth}) that $\Hom(\cF, Ri^!\cG) = 0$, and
therefore $\Hom(i_*\cF, \cG) = 0$.   The latter implies that $i_*\cF \in
\dsl Xw$.  Similarly, if $\cF \in \dsg Zw$, then $\Hom(Li^*\cG, \cF) =
\Hom(\cG, i_*\cF) = 0$ for all $\cG \in \dsml X{w-1}$, so $i_*\cF \in \dsg
Xw$.  For the opposite inclusion, given an object $\cF \in \dgb Z$, form
the distinguished triangle
\[
i_*\beta_{\le w}\cF \to i_*\cF \to i_*\beta_{\ge w+1}\cF \to
\]
in $\dgb X$.  By the reasoning above, we have $i_*\beta_{\le w}\cF \in \dsl Xw$
and $i_*\beta_{\ge w+1} \in \dsg X{w+1}$, so the first and last terms
above must be the baric truncations of $i_*\cF$:
\[
i_*\beta_{\le w}\cF \cong \beta_{\le w}i_*\cF
\qquad\text{and}\qquad
i_*\beta_{\ge w+1}\cF \cong \beta_{\ge w+1}i_*\cF.
\]
Thus, if $i_*\cF \in \dsl Xw$, then $\beta_{\ge w+1}i_*\cF = i_*\beta_{\ge
w+1}\cF = 0$.  Since $i_*$ is faithful, this implies that $\beta_{\ge
w+1}\cF = 0$, so that $\cF \in \dsl Zw$.  The same argument shows that
$i_*\cF \in \dsg Xw$ implies that $\cF \in \dsl Xw$.

Next, assume the categories~\eqref{eqn:hered} constitute a baric
structure on $Z$.  We will show that this baric structure is induced from
the one on $X$.  If $\cF \in \dsml Xw$, then $\Hom(\cF, i_*\cG) = 0$ for
all $\cG \in \dsg Z{w+1}$ by Lemma~\ref{lem:unbdd-orth}, so
$\Hom(Li^*\cF,\cG) = 0$, and hence $Li^*\cF \in \dsml Zw$.  Similarly, if
$\cF \in \dspg Xw$, then $\Hom(i_*\cG,\cF) = \Hom(\cG, Ri^!\cF) = 0$ for
all $\cG \in \dsl Z{w-1}$, so $Ri^!\cF \in \dspg Zw$.  Thus, $Li^*$ is
right baryexact, and $Ri^!$ is left baryexact, as desired. 

We turn now to open subschemes.  Suppose $\schemebaric U$ is an induced
baric structure on an open subscheme $j: U \hto X$.  In view of the
equalities~\eqref{eqn:pm-bdd}, the definition of ``induced'' implies that
$j^*: \dgb X \to \dgb U$ is baryexact.  In other words, if
$\cF_1 \in \dsl Xw$, then $j^*\cF_1 \in \dsl Uw$, and if $\cF_1 \in \dsg
Xw$, then $j^*\cF_1 \in \dsg Uw$.  Conversely, if $\cF \in \dsl Uw$, then
there exists some object $\cF' \in \dgb X$ such that $j^*\cF' \cong \cF$. 
Form the distinguished triangle $\beta_{\le w}\cF' \to \cF' \to \beta_{\ge
w+1}\cF' \to$, and apply $j^*$ to it.  We know that $j^*\beta_{\le w}\cF'
\in \dsl Uw$ and that $j^*\beta_{\ge w+1}\cF' \in \dsg U{w+1}$.  Since
$j^*\cF' \cong \cF$, we see from the triangle
\[
j^*\beta_{\le w}\cF' \to \cF \to j^*\beta_{\ge w+1}\cF' \to
\]
that $j^*\beta_{\ge w+1}\cF' \cong \beta_{\ge w+1}\cF = 0$, and hence that
$\cF \cong j^*\beta_{\le w}\cF'$.  Thus, setting $\cF_1 = \beta_{\le
w}\cF'$, we have found an $\cF_1 \in \dsl Xw$ such that $j^*\cF_1 \cong
\cF$.  The argument for $\dsg Uw$ is similar.

Finally, assume the categories~\eqref{eqn:local} constitute a baric
structure on $U$.  We must show that this baric structure is induced. 
Clearly, $j^*$ is baryexact as a functor of bounded derived categories
$\dgb X \to \dgb U$.  Since $j^*$ is also exact, it commutes with
truncation and cohomology functors, and it takes $\dgbg Xw$ to $\dgbg
Uw$.  It follows from these observations that it takes $\dsml Xw$ to $\dsml
Uw$ and $\dspg Xw$ to $\dspg Uw$.
\end{proof}

\begin{lem}\label{lem:hlr-baric-res}
Let $j: U \hto X$ be the inclusion of an open subscheme, and let $i: Z \hto
X$ be the inclusion of a closed subscheme.  Assume that $U$ and $Z$ are equipped with baric structures induced from one on $X$.  Then:
\begin{enumerate}
\item $j^*$ takes $\dsml Xw$ to $\dsml Uw$ and $\dspg Xw$ to
$\dspg Uw$.
\item $Li^*$ takes $\dsml Xw$ to $\dsml Zw$.
\item $Ri^!$ takes $\dspg Xw$ to $\dspg Zw$.
\item $i_*$ takes $\dsml Zw$ to $\dsml Xw$ and $\dspg Zw$ to
$\dspg Xw$.
\end{enumerate}
\end{lem}
\begin{proof}
Parts~(1), (2), and~(3) hold by definition.

(4)~We saw in the proof of Lemma~\ref{lem:induced}
that as a functor of bounded derived categories $\dgb Z \to \dgb X$, $i_*$
is baryexact.  Since $i_*$ is also an exact functor, we have $h^k(i_*\cF)
\cong i_*h^k(\cF)$ for any $\cF \in \dgm Z$.  Thus, if $\cF \in \dsml
Zw$, we have $h^k(i_*\cF) \in \dsl Xw$ for all $k$; in other words,
$i_*\cF \in \dsml Xw$.  On the other hand, suppose $\cF \in \dspg Zw$. 
Since $i_*$ is exact and baryexact on $\dgb Z$, we have $i_*\beta_{\le
w-1}\tau^{\le k}\cF \cong \beta_{\le w-1}\tau^{\le k}i_*\cF$.  Moreover,
the fact that $\beta_{\le w-1}\tau^{\le k}\cF \in \dgbg Z{k+1}$ for all
$k$ implies that $i_*\beta_{\le w-1}\tau^{\le k}\cF \in \dgbg X{k+1}$ for
all $k$.  Thus, $i_*\cF \in \dspg Xw$.
\end{proof}

\begin{lem}\label{lem:closed-hered}
Let $\schemebaric X$ be a hereditary baric structure on $X$, and let $i: Z
\hto X$ be the inclusion of a closed subscheme.  The induced baric
structure on $Z$ is also hereditary.
\end{lem}
\begin{proof}
Let $\kappa: Y \hto Z$ be a closed subscheme of $Z$.  We must show that $Y$
admits a baric structure
induced from the one on $Z$.  In fact, we claim that the baric structure
on $Y$ induced from the on $X$ (via $i \circ \kappa: Y \hto X$) has the
desired property. Suppose $\cF \in \dsml Zw$.  If $L\kappa^*\cF \notin
\dsml Yw$, then there is some $\cG \in \dsg Y{w+1}$ such that
$\Hom(L\kappa^*\cF,\cG)
\ne 0$. Then $\Hom(\cF, \kappa_*\cG) \ne 0$ and, because $i_*$ is
faithful,
$\Hom(i_*\cF, i_*\kappa_*\cG) \ne 0$.  But this is impossible,
because according to Lemma~\ref{lem:hlr-baric-res}, $i_*\cF \in \dsml
Xw$ and $(i\circ \kappa)_*\cG \in \dsg X{w+1}$.  Thus, $L\kappa^*\cF \in
\dsml Yw$.  Similarly, if $\cF \in \dspg Zw$, a consideration of $\Hom(\cG,
R\kappa^!\cF)$ and $\Hom(i_*\kappa_*\cG, i_*\cF)$ for $\cG \in \dsl
Y{w-1}$ shows that $R\kappa^!\cF \in \dspg Yw$.  Thus, $L\kappa^*$ is right
baryexact and $R\kappa^!$ is left baryexact, so the baric structure on $Y$
induced from the one on $X$ is also induced from the one on $Z$.  The
induced baric structure on $Z$ is therefore hereditary.
\end{proof}

\subsection{Properties of HLR baric structures}
\label{sect:hlrprop}

In this section, we prove three useful results about HLR baric structures. 
First, we prove that the HLR property is inherited by induced baric
structures on subschemes.  Next, we prove an additional rigidity property
for nilpotent thickenings of closed subschemes.  Finally,
we prove a ``gluing theorem'' that states that an HLR baric structure is
determined by the baric structures it induces on a closed subscheme and the
complementary open subscheme.  It should be noted that the proofs of these
results depend on Theorem~\ref{thm:hlr}.

\begin{thm}\label{thm:hlr-induced}
Suppose $X$ is endowed with an HLR baric structure.  Every locally closed
subscheme $\kappa: Y \hto X$
admits a unique induced baric structure.  Moreover, this baric
structure is also HLR.
\end{thm}
\begin{proof}
We have already seen the uniqueness of the induced baric structure in the
case of open or closed subschemes, in Lemma~\ref{lem:induced}.  For a
general locally closed subscheme, let us factor the inclusion map $\kappa:
Y \to X$ as a closed imbedding $i: Y \hto U$ followed by an open imbedding
$j: U \hto X$.  Then $U$ acquires a unique induced hereditary baric
structure from the baric structure on $X$, and it in turn induces a unique
baric structure on its closed subscheme $Y$.  This baric structure is also
induced from the one on $X$: clearly, $L\kappa^* = Li^* \circ j^*$ is
right baryexact, and $R\kappa^! = Ri^! \circ j^*$ is left baryexact.

To show that this is the unique baric structure on $Y$ induced from the
one on $X$, we must show that the baryexactness assumptions on $L\kappa^*$
and $R\kappa^!$ imply the same conditions on $Li^*$ and $Ri^!$.  (It then
follows that any baric structure induced from the one on $X$ is actually
induced from the one on $U$.)  Suppose $\cF \in \dsml Uw$, and consider
a distinguished triangle of the form
\[
Li^*\tau^{\le k-1}\cF \to Li^*\cF \to Li^*\tau^{\ge k}\cF \to.
\]
Since $Li^*\tau^{\le k-1}\cF \in \dgml Y{k-1}$, we see that $h^k(Li^*\cF)
\cong h^k(Li^*\tau^{\ge k}\cF)$.  Now, $\tau^{\ge k}\cF$ is an object in
$\dsl Uw$, so there exists an object $\cF_1 \in \dsl Xw$ such that
$j^*\cF_1 \cong \tau^{\ge k}\cF$.  By assumption, $L\kappa^*\cF_1 \in
\dsml Yw$.  But $L\kappa^*\cF \cong Li^*\tau^{\ge k}\cF$, so we conclude
that $h^k(Li^*\tau^{\ge k}\cF) \cong h^k(Li^*\cF) \in \dsl Yw$.  Thus,
$Li^*\cF \in \dsml Yw$.  

On the other hand, suppose that $\cF \in \dspg Uw$, and consider a
distinguished triangle of the form
\[
Ri^!\tau^{\le k}\cF \to Ri^!\cF \to Ri^!\tau^{\ge k+1}\cF \to.
\]
Since $Ri^!\tau^{\ge k+1}\cF \in \dgpg Y{k+1}$, we see that $\tau^{\le
k}Ri^!\cF \cong \tau^{\le k}Ri^!\tau^{\le k}\cF$.  Next, consider the
distinguished triangle
\[
Ri^!\beta_{\le w-1}\tau^{\le k}\cF \to Ri^!\tau^{\le k}\cF \to Ri^!
\beta_{\ge w}\tau^{\le k}\cF \to.
\]
By assumption, $\beta_{\le w-1}\tau^{\le k}\cF \in \dgbg U{k+2}$, so $Ri^!\beta_{\le w-1}\tau^{\le k}\cF \in \dgpg Y{k+2}$.  It
follows that $\tau^{\le k}Ri^!\tau^{\le k}\cF \cong
\tau^{\le k}Ri^!\beta_{\ge w}\tau^{\le k}\cF$.
Now, $\beta_{\ge w}\tau^{\le k}\cF \in \dsg Uw$, so there is some $\cF_1
\in \dsg Xw$ such that $j^*\cF_1 \cong \beta_{\ge w}\tau^{\le k}\cF$. 
Since $R\kappa^!\cF_1$ belongs to $\dspg Yw$ by assumption, we have
$\beta_{\le w-1}\tau^{\le k}R\kappa^!\cF_1 \in \dgbg Y{k+2}$.  But we also
have $R\kappa^!\cF_1 \cong Ri^!\beta_{\ge w}\tau^{\le k}\cF$, and from the
chain of isomorphisms
\[
\tau^{\le k}Ri^!\cF \cong \tau^{\le k}Ri^!\tau^{\le k}\cF \cong
\tau^{\le k}Ri^!\beta_{\ge w}\tau^{\le k}\cF \cong
\tau^{\le k}R\kappa^!\cF_1,
\]
we see that $\beta_{\le w-1}\tau^{\le k}Ri^!\cF \in \dgbg Y{k+2}$.  Thus,
$Ri^!\cF \in \dspg Yw$.  We now conclude that any baric structure on $Y$
induced from the one on $X$ is also induced from the one on $U$, and is
therefore uniquely determined.

To show that the induced baric structure on a locally closed subscheme is
HLR, it suffices, by Theorem~\ref{thm:hlr}, to show that it is hereditary.
In the case of a closed subscheme, this was done in
Lemma~\ref{lem:closed-hered}, and in the case of an open subscheme, there
is nothing to prove: this property is part of the definition of ``local.'' 
The assertion then follows for a general locally closed subscheme, since,
by construction, the induced baric structure on such a subscheme is
obtained by first passing to an open subscheme, and then to a closed
subscheme of that.
\end{proof}

Next, we turn to nilpotent thickenings of a closed subscheme.

\begin{prop}\label{prop:rigid}
Suppose $X$ is endowed with an HLR baric structure, and let $Z
\overset{t}{\hto} Z_1 \hto X$ be a sequence of closed subschemes of $X$
with the same underlying topological space.  Then:
\begin{enumerate}
\item For $\cF \in \dgm {Z_1}$, $\cF \in \dsml {Z_1}w$ if and only if
$Lt^*\cF \in \dsml Zw$.
\item For $\cF \in \dgp {Z_1}$, $\cF \in \dspg {Z_1}w$ if and only if
$Rt^!\cF \in \dspg Zw$.
\end{enumerate}
\end{prop}
\begin{proof}
If $\cF \in \dsml {Z_1}w$, it is obvious that $Lt^*\cF \in \dsml Zw$,
since the baric structure on $Z$ is induced from that on $Z_1$. 
Conversely, suppose $\cF \in \dgm {Z_1}$ and $Lt^*\cF \in \dsml Zw$.  Then
$\Hom(\cF, t_*\cG) \cong \Hom(Lt^*\cF,\cG) = 0$ for all $\cG \in \dsg
Z{w+1}$.  But by the definition of ``rigid,'' $\dsg {Z_1}{w+1}$ is
generated by objects of the form $t_*\cG$ with $\cG \in \dsg Z{w+1}$, so it
follows that $\Hom(\cF,\cG') = 0$ for all $\cG' \in \dsg {Z_1}{w+1}$, and
hence that $\cF \in \dsml {Z_1}w$.  The proof of part~(2) is entirely
analogous and will be omitted.
\end{proof}

Finally, we prove a ``gluing theorem'' for HLR baric structures.

\begin{thm}
Suppose $X$ is endowed with an HLR baric structure.  Let $i:Z \hookrightarrow X$ be a closed subscheme of $X$,
and let $j:U \hookrightarrow X$ be its open complement.  Endow $U$ and $Z$ with the baric structures induced from that on $X$.  Then we have
\begin{align*}
\dsl Xw &= \{ \cF \in \dgb X \mid \text{$j^*\cF \in \dsl Uw$ and $Li^*\cF
\in \dsml Zw$} \}, \\
\dsg Xw &= \{ \cF \in \dgb X \mid \text{$j^*\cF \in \dsg Uw$ and $Ri^!\cF
\in \dspg Zw$} \}.
\end{align*}
In particular, there is a unique HLR baric structure on $X$ which induces
the baric structures $\schemebaric U$ and $\schemebaric Z$ on $U$ and $Z$.
\end{thm}
\begin{proof}
If $\cF \in \dsl Xw$, then $j^*\cF \in \dsl Uw$ and $Li^*\cF \in \dsml Zw$ by the definition of the induced baric structure.  For the other direction, suppose that $j^*
\cF \in \dsl Uw$ and $Li^* \cF \in \dsml Zw$.  We will prove that $\cF \in
\dsl Xw$ by showing $\Hom(\cF,\cG) = 0$ for all $\cG \in \dsg X{w+1}$.  

Fix $\cG \in \dsg X{w+1}$.  We have an exact sequence
\[
\lim_{\substack{\to \\ {Z_1}}} \Hom(i_{Z_1*}Li_{Z_1}^*\cF,\cG) \to
\Hom(\cF,\cG) \to \Hom(j^*\cF,j^*\cG),
\]
where the limit runs over nilpotent thickenings of $Z$.  
(See, for instance,~\cite[Proposition~2 and Lemma~3(a)]{bez:pcs} for an
explanation of this exact sequence.)  We have $j^* \cF \in \dsl Uw$ and
$j^*\cG \in \dsg U{w+1}$, and by Lemma~\ref{lem:hlr-baric-res}, we have
$i_{Z_1 *} Li_{Z_1}^* \cF \in \dsml Xw$ so the first and third terms
vanish.  We conclude that $\Hom(\cF,\cG)$ also vanishes.  The argument for
$\dsg Xw$ is similar.
\end{proof}

\subsection{Proof of Theorem~\ref{thm:hlr}}
\label{sect:hlrproof}

In this section, we will prove that hereditary baric structures are
automatically also local and rigid.  We begin with a result about baric
truncation functors with respect to a hereditary baric structure.  If $X$
is endowed with a hereditary baric structure, and $\cF \in \dgb X$ is
actually supported on some closed subscheme $i: Z \hto X$, then the baric
truncations of $\cF$ are obtained by taking baric truncations in the
induced baric structure on $Z$, and then pushing them forward by $i_*$.  In
other words, hereditary baric structures have the property that baric
truncation functors preserve support.  More precisely:

\begin{prop}
\label{prop:settheorysupport}
Let $\schemebaric X$ be a hereditary baric structure on $X$.  Then
\begin{enumerate}
\item If $\cF \in \dgb X$ has set-theoretic support on a closed set $Z
\subset X$, then so do $\beta_{\leq w} \cF$ and $\beta_{\geq w} \cF$.
\item If a morphism $u:\cF \to \cG$ in $\dgb X$ has set-theoretic support
on $Z$, in the sense that $u \vert_{X \ssm Z} = 0$, then so do $\beta_{\leq
w}(u)$ and $\beta_{\geq w}(u)$.
\end{enumerate}
\end{prop}

\begin{proof}
If $\cF$ is set-theoretically supported on $Z$ then there is a subscheme
$i:Z_1 \hookrightarrow X$ of $X$, whose underlying closed set is $Z$, such
that $\cF = i_* \cF'$ for some $\cF' \in \dgb {Z_1}$.  Form the distinguished triangle
\[
\beta_{\leq w} \cF' \to \cF' \to \beta_{\geq w+1} \cF' \to.
\]
By Lemma~\ref{lem:induced}, we have that $i_* \beta_{\leq w}\cF' \in
\dsl Xw$ and $i_* \beta_{\geq w+1}\cF' \in \dsg X{w+1}$.  Since we have a
distinguished triangle
\[
i_* \beta_{\leq w} \cF' \to \cF \to i_* \beta_{\geq w+1} \cF'\to,
\]
we must have $i_* \beta_{\leq w} \cF' \cong \beta_{\leq w} \cF$ and $i_*
\beta_{\geq w+1} \cF' \cong \beta_{\geq w+1} \cF$.  In particular these objects
are set-theoretically supported on $Z$, proving the first assertion.

To prove the second assertion, consider the exact sequence
\[
\lim_{\substack{\to \\ Z'}} \Hom(\cF, i_{Z'*}Ri^!_{Z'}\cG) \to
\Hom(\cF,\cG) \to \Hom(\cF|_{X\ssm Z},\cG|_{X\ssm Z}),
\]
where $i_{Z'}: Z' \hto X$ ranges over all closed subscheme structures on
$Z$.  By assumption, $u \in \Hom(\cF,\cG)$ vanishes upon restriction to $X
\ssm Z$, so we see from the exact sequence above that it must factor
through $i_{Z'*}Ri^!_{Z'}\cG \to \cG$ for some closed subscheme structure
$i_{Z'}: Z' \hto X$ on $Z$.  Now, $i_{Z'*}Ri^!_{Z'}\cG$ is in general an
object of $\dgp X$, but since $\cF$ lies in $\dgb X$, any morphism $\cF \to
i_{Z'*}Ri^!_{Z'}\cG$ factors through $\tau^{\le n}i_{Z'*}Ri^!_{Z'}\cG$ for
sufficiently large $n$.  It follows that $\beta_{\le w}(u)$ and $\beta_{\ge
w}(u)$ factor through $\beta_{\le w}\tau^{\le n}i_{Z'*}Ri^!_{Z'}\cG$ and
$\beta_{\ge w}\tau^{\le n}i_{Z'*}Ri^!_{Z'}\cG$, respectively.  These objects
have set-theoretic support on $Z$ by the first part of the proposition, so
$\beta_{\le w}(u)$ and $\beta_{\ge w}(u)$ have set-theoretic support on $Z$
as well, as desired.
\end{proof}

We may use this fact to prove the following:

\begin{thm}\label{thm:openlocal}
Every hereditary baric structure is local.
\end{thm}

We will prove this theorem over the course of the following three
propositions.  Recall from Lemma~\ref{lem:induced} that in a local baric
structure, the induced baric structures on open subschemes necessarily have
the form given in the proposition below.

\begin{prop}\label{prop:hered-open}
Let $\schemebaric X$ be a hereditary baric structure on $X$, and let $U$ be
an open subschme of $X$.  For any $w \in \Z$, define full subcategories of
$\dgb U$ as follows:
\begin{align*}
\dsl Uw &= \{ \cF \in \dgb U \mid \text{$\cF \cong j^*\cF_1$ for some
$\cF_1 \in \dsl Xw$} \}, \\
\dsg Uw &= \{ \cF \in \dgb U \mid \text{$\cF \cong j^*\cF_1$ for some
$\cF_1 \in \dsg Xw$} \}.
\end{align*}
Then $\dsl Uw$ and $\dsg Uw$ are thick subcategories of $\dgb U$.
\end{prop}
\begin{proof}
Suppose that $\cF$ and $\cG$ belong to $\dsl Uw$, so that there exist
$\cF_1$ and $\cG_1$ in 
$\dsl Xw$ with $\cF_1 \vert_U \cong \cF$ and $\cG_1 \vert_U \cong \cG$. 
Since $\dgb U$ is a localization of $\dgb X$, we may find for every
morphism $u:\cF \to \cG$ an object $\cG_2 \in \dgb X$
and a diagram
$\cF_1 \to \cG_2 \leftarrow \cG_1$
such that $(\cG_2 \leftarrow \cG_1)\vert_U$ is an isomorphism, and the
composition
\[
\cF \cong \cF_1 \vert_U \to \cG_2 \vert_U \cong \cG_1 \vert_U \cong
\cG
\]
coincides with $u$.  We claim that the diagram
\[
\beta_{\leq w} \cF_1 \to \beta_{\leq w} \cG_2 \leftarrow \beta_{\leq w}
\cG_1
\]
has the same property.  In that case, the cone on the composition $\cF_1
\cong \beta_{\leq w} \cF_1 \to \beta_{\leq w} \cG_2$ belongs to $\dsl Xw$,
which shows that the cone on $u:\cF \to \cG$ belongs to $\dsl Uw$.  To
prove the claim, note that the cone on the map $\cG_1 \to \cG_2$ is
set-theoretically supported on the closed set $X \ssm U$, and since
the baric structure $\schemebaric X$ is hereditary, the same must be true
for the cone on
$\beta_{\leq w} \cG_1 \to \beta_{\leq w} \cG_2$; in particular the
restriction of the latter map to $U$ is an isomorphism.  

We have shown that the $\dsl Uw \subset \dgb U$ is a triangulated subcategory.  To show that it is thick we have to show that it is also closed under summands -- \emph{i.e.} that if $\cF \oplus \cG \in \dsl Uw$ then $\cF$ and $\cG$ also belong to $\dsl Uw$.   Thus suppose that $\cF \oplus \cG$ belongs to $\dsl Uw$.  Since $\dgb U$ is a localization of $\dgb X$, we may find a triangle
$$\cF_1 \to \cH \to \cG_1 \to$$
whose restriction to $U$ is isomorphic to the triangle
$$\cF \to \cF \oplus \cG \to \cG \to $$
In particular the map $\cG_1 \to \cF_1[1]$ is set-theoretically supported on $X \ssm U$, so by proposition \ref{prop:settheorysupport} the same must be true of $\beta_{\leq w} \cG_1 \to \beta_{\leq w} \cF_1$.  From the diagram
$$
\xymatrix{
\beta_{\leq w} \cF_1 \ar[r] \ar[d] & \cH \ar[r] \ar[d] & \beta_{\leq w} \cG_1 \ar[r] \ar[d]& \\
\cF_1 \ar[r] \ar[d] & \cH \ar[r] \ar[d] & \cG_1 \ar[r] \ar[d]& \\
\beta_{\geq w+1} \cF_1 \ar[r] & 0 \ar[r]& \beta_{\geq w+1} \cG_1 \ar[r] & \\
}
$$
whose rows and columns are distinguished triangles, we see that $\beta_{\geq w+1} \cG_1 \to \beta_{\geq w+1} \cF_1$ is an isomorphism.  But since this morphism has set-theoretic support on $X - U$ the objects $\beta_{\geq w+1} \cF_1$ and $\beta_{\geq w+1} \cG_1$ must have set-theoretic support on $X- U$ which implies there are isomorphisms $\beta_{\leq w} \cF_1 \vert_U \cong \cF$ and $\beta_{\leq w} \cG_1 \vert_U \cong \cG$.  Thus $\cF$ and $\cG$ belong to $\dsl Uw$.

A similar proof shows that the subcategories $\dsg Uw$ are thick.
\end{proof}

\begin{prop}
Let $\schemebaric X$ be a hereditary baric structure on $X$, let $U$ be an
open subscheme of $X$, and 
let $\schemebaric U$ be as in Proposition~\ref{prop:hered-open}.  Then
$\schemebaric U$ is a baric structure on $\dgb U$, compatible with the
standard $t$-structure.
\end{prop}
\begin{proof}
It is clear that $\dsl Uw \subset \dsl U{w+1}$ and $\dsg Uw \supset \dsg
U{w+1}$ and that $\dgb U = \dsl Uw * \dsg U{w+1}$.  If $\cF \in \dsl Uw$
and $\cG \in \dsg U{w+1}$, then we have an exact sequence
\[
\Hom(\cF_1,\cG_1) \to \Hom(\cF,\cG) \to
\varinjlim_{i:Z \hto X} \Hom(i_*Li^* \cF_1, \cG_1[1]) \to
\]
where $\cF_1$ is an extension of $\cF$ to $\dsl Xw$, $\cG_1$ is an
extension of $\cG$ to $\dsg X{w+1}$, and $i:Z \hookrightarrow X$ runs over
all subscheme structures on $X \ssm U$.  The first term above
vanishes automatically, and each of the terms $\Hom(i_* Li^*
\cF_1,\cG_1[1])$ vanishes because, by Lemma~\ref{lem:hlr-baric-res},
$i_*Li^*\cF_1 \in \dsml Xw$. Thus, $\Hom(\cF,\cG) = 0$ and
$\schemebaric U$ is a baric structure on $\dgb U$.

By assumption the baric structure $\schemebaric X$ is compatible with the standard $t$-structure on $\dgb X$.  Thus if $\cF_1$ belongs to $\dsl Xw$ then so
do $\tau^{\leq n} \cF_1$ and $\tau^{\geq n} \cF_1$.  The objects $\cF_1 \vert_U$,
$(\tau^{\leq n} \cF_1) \vert_U \cong \tau^{\leq n} (\cF_1 \vert_U)$ and
$(\tau^{\geq n} \cF_1)\vert_U \cong \tau^{\geq n} (\cF_1 \vert_U)$ therefore all
belong to $\dsl Uw$.  Similarly, we have $(\beta_{\leq w} \cF_1)\vert_U
\cong \beta_{\leq w} (\cF_1\vert_U)$ and $(\beta_{\geq w} \cF_1)\vert_U
\cong \beta_{\geq w} (\cF_1 \vert_U)$ so that the baric truncation
functors preserve $\dgb U^{\geq 0}$.  Thus the baric structure
$\schemebaric U$ is compatible with the standard $t$-structure on $\dgb U$.
\end{proof}

\begin{prop}
Let $\schemebaric X$ be a hereditary baric structure on $X$, and let $U$ be
an open subscheme of $X$.  Then the collection of categories $\schemebaric
U$ defined in Proposition~\ref{prop:hered-open} constitute a
hereditary baric structure on $U$.
\end{prop}
\begin{proof}
Using Lemma~\ref{lem:induced} and the previous proposition, we know that
the baric structure $\schemebaric U$ is induced from the one on $X$.  It
remains only to show that this baric structure is hereditary.  Let $i:Y
\hookrightarrow U$ be a closed subscheme of $U$.  By
Lemma~\ref{lem:induced}, we must prove that the following categories
constitute a baric structure on $Y$:
\begin{align*}
\dsl Yw &= \{ \cF \in \dgb Y \mid 
\text{$i_*\cF \cong \cF_1|_U$ for some $\cF_1 \in \dsl Xw$} \}, \\
\dsg Yw &= \{ \cF \in \dgb Y \mid 
\text{$i_*\cF \cong \cF_1|_U$ for some $\cF_1 \in \dsg Xw$} \}.
\end{align*}
Let $\overline{Y}$ be the closure of $Y$ in $X$, and let $i_1: \overline Y
\hto X$ be the inclusion map, so that we have a commutative square of
inclusions
\[
\xymatrix@=12pt{
Y \ar@{^{(}->}[r]^i \ar@{^{(}->}[d] & U \ar@{^{(}->}[d] \\
\overline{Y} \ar@{^{(}->}[r]^{i_1} & X}
\]

By definition, the hereditary baric structure on $X$ induces a baric
structure on $\overline{Y}$.  This baric structure is itself hereditary,
by Lemma~\ref{lem:closed-hered}.  Thus,
by the previous proposition, the baric structure on $\overline Y$ induces
one on its open subscheme $Y$.  This is
given by
\begin{align*}
(\dsl Yw)' &= \{ \cF \mid
\text{$\cF \cong \cF_2|_Y$ for some $\cF_2 \in \dgb {\overline Y}$ with
$i_{1*}\cF_2 \in \dsl Xw$} \}, \\
(\dsg Yw)' &= \{ \cF \mid
\text{$\cF \cong \cF_2|_Y$ for some $\cF_2 \in \dgb {\overline Y}$ with
$i_{1*}\cF_2 \in \dsg Xw$} \}.
\end{align*}
It suffices now to show that $\dsl Yw = (\dsl Yw)'$ and $\dsg Yw = (\dsg
Yw)'$.  If $\cF \in \dgb Y$ is such that we may find $\cF_2 \in
\dgb{\overline{Y}}$ with $\cF_2 \vert_Y \cong \cF$ and $i_* \cF_2 \in \dsl
Xw$ then $\cF_1 := i_{1*} \cF_2$ has the property that $\cF_1 \vert_U \cong
i_* \cF$.  Thus, $(\dsl Yw)' \subset \dsl Yw$.  To show the reverse
inclusion, let $\cF \in \dgb Y$ and $\cF_1 \in \dsl Xw$ be such that $\cF_1
\vert_U \cong i_* \cF$, and let $\cF_2' \in \dgb {\overline{Y}}$ be such
that there exists a map $i_{1*} \cF_2' \to \cF_1$ which is an isomorphism
over $U$.  Then $i_{1*} \beta_{\leq w} \cF_2' \to \cF_1$ is also an
isomorphism over $U$, and $\cF_2 := \beta_{\leq w} \cF_2'$ has the property
that $\cF_2 \vert_Y \cong \cF$ and $i_{1*} \cF_2 \in \dsl Xw$.  Thus,
$(\dsl Yw)' = \dsl Yw$.  A similar argument shows that $(\dsg Yw)' = \dsg
Yw$.
\end{proof}

Let us finally show that hereditary baric structures are rigid.

\begin{prop}
\label{prop:HimpliesR1}
Let $\schemebaric X$ be a hereditary baric structure on $X$.  Then
$\schemebaric X$ is rigid.
\end{prop}
\begin{proof}
Let $Z$ be a subscheme of $X$ and let $Z_1$ be a nilpotent thickening of
$Z$ in $X$, and write $t$ for inclusion of $Z$ into $Z_1$.
If $\cF$ is a bounded chain complex of coherent sheaves on $Z_1$, then we
may find a filtration of $\cF$ by subcomplexes $\cF_k$ whose subquotients
are scheme-theoretically supported on $Z$.  Thus in $\dgb {Z_1}$ we may
find a sequence of objects and maps 
\[
0 = \cF_0 \to \cF_1 \to \cF_2 \to \cdots \to \cF_n = \cF
\]
such that the cone on $\cF_{k-1} \to \cF_k$ is of the form $t_* \cG_k$. 
Now suppose $\cF$ belongs to $\dsl {Z_1}w$.  Then we may apply $\beta_{\leq
w}$ to the sequence to obtain
\[
0 = \beta_{\leq w} \cF_0 \to \beta_{\leq w} \cF_1 \to \cdots \to
\beta_{\leq w} \cF_n = \cF
\]
and distinguished triangles
\[
\beta_{\leq w} \cF_{k-1} \to \beta_{\leq w} \cF_k \to \beta_{\leq w} t_*
\cG_k \to.
\]
It follows from Lemma~\ref{lem:induced} that the object $\beta_{\leq w} t_*
\cG_k$ is isomorphic to $t_* \beta_{\leq w} \cG_k$.  Thus, $\cF$ is in the
thick closure of the image of $\dsl Zw$ under $t_*$.  A similar proof gives
the same result for $\dsg {Z_1}w$.
\end{proof}

This completes the proof of Theorem~\ref{thm:hlr}.

\section{Background on $s$-structures and Staggered Sheaves}
\label{sect:stagt}

In this section, we review the $t$-structures on derived categories of
equivariant coherent sheaves that were introduced in~\cite{a}.  (They were
called ``staggered $t$-structures'' in {\it loc.~cit.}; in
Section~\ref{sect:stag2}, we will prove that they usually arise by the
staggering construction of Definition~\ref{defn:stag}.)  These
$t$-structures depend on two auxiliary data: an \emph{$s$-structure}, and a
\emph{perversity function}.  After fixing notation, we briefly recall some
facts about these objects, and we then describe the $t$-structures
themselves.  We will also prove a few useful lemmas about these objects.

As before, let $X$ be a scheme of finite type over a noetherian base
scheme, acted on by an affine group scheme $G$ over the same base.  We
adopt the additional assumptions that the base scheme admits a dualizing
complex in the sense of~\cite[Chap.~V]{har}, and that the category $\cg X$
has enough locally free objects.  It follows
(see~\cite[Proposition~1]{bez:pcs}) that $X$ admits an equivariant
dualizing complex.  Fix one, and denote it $\omega_X \in \dgb X$.  Next,
let $\D = \cRHom(\cdot, \omega_X)$ denote the equivariant
Serre--Grothendieck duality functor.  Let $X^\gen$ denote the set of
generic points of $G$-invariant subschemes of $X$, and for any $x \in
X^\gen$, we denote by $\barGx$ the smallest $G$-stable closed subset of
$X$.  (We do not usually regard $\barGx$ as having a fixed subscheme
structure.)

For any point $x \in
X^\gen$ and any closed subscheme structure $i: Z \hto X$ on $\barGx$, there
is an open subscheme $V \subset Z$ such that $Ri^!\omega_X|_V$ is
concentrated in a single degree in $\dgb V$.  Let $\cod \barGx$ be the
unique integer such that $h^{\cod \barGx}(Ri^!\omega_X|_V) \ne 0$.  This
number is independent of the choice of closed subscheme structure $i: Z
\hto X$ and of open subscheme $V \subset Z$.  If $X$ is, say, an
equidimensional scheme of finite type over a field, $\omega_X$ may be
normalized so that $\cod \barGx$ is the ordinary (Krull) codimension of
$\barGx$.

An \emph{$s$-structure} on the scheme $X$ is a pair of collections of full
subcategories $(\{\cgl Xw\}, \{\cgg Xw\})_{w \in \Z}$ of $\cg X$ satisfying
a list of ten axioms, called (S1)--(S10) in~\cite{a}.  We will not review
all the axioms here, but we do recall some of the key properties of
$s$-structures:
\begin{itemize}
\item Each $\cgl Xw$ is a Serre subcategory, and each $\cgg Xw$ is closed
under extensions and subobjects.
\item $\cgg Xw$ is the right orthogonal to $\cgl X{w-1}$.
\item Each sheaf $\cF$ contains a unique maximal subsheaf in $\cgl Xw$,
denoted $\sigma_{\le w}\cF$.  The quotient $\sigma_{\ge w+1}\cF \cong
\cF/\sigma_{\le w}\cF$ is the largest quotient of $\cF$ in $\cgg X{w+1}$.
\item An $s$-structure on $X$ induces $s$-structures on all locally closed
subschemes of $X$.
\end{itemize}
Assume henceforth that $X$ is equipped with a fixed $s$-structure.  Given a
point $x \in X^\gen$ and a closed subscheme structure $i: Z \hto X$ on
$\barGx$, choose an open subscheme $V \subset Z$ such that
$Ri^!\omega_X|_V$
is concentrated in degree $\cod \barGx$.  There is a unique integer, called
the \emph{altitude} of $\barGx$ and denoted $\alt\barGx$, such that
\[
Ri^!\omega_X|_V[\cod\barGx] \in \cgl V{\alt\barGx} \cap \cgg V{\alt\barGx}.
\]
Again, $\alt\barGx$ is independent of the choice of $i$ and $V$.

The \emph{staggered codimension} of $\barGx$ is defined by
\[
\scod \barGx = \alt \barGx + \cod \barGx.
\]
A (\emph{staggered}) \emph{perversity function} is a function $p: X^\gen
\to \Z$ such that
\[
0 \le p(x) - p(y) \le \scod \barGx - \scod \barGy
\qquad\text{if $x \in \barGy$.}
\]
Given a perversity $p: X^\gen \to \Z$, the function $\bar p: X^\gen \to \Z$
given by
\[
\bar p (x) = \scod \barGx - p(x)
\]
is also a perversity function, known as the \emph{dual perversity}.  Given
a staggered perversity function $p$, we define a full subcategory of $\dgm
X$ by
\[
\p \dgml X0 = \left\{ \cF \,\Bigg| 
\begin{array}{c}
\text{for any $x \in X^\gen$, any closed subscheme structure} \\
\text{$i: Z\hto X$ on $\barGx$, and any $k \in \Z$, there is a dense open}
\\
\text{subscheme $V \subset Z$ such that $h^k(Li^*\cF)|_V \in \cgl
V{p(x)-k}$}
\end{array} \right\},
\]
and a full subcategory of $\dgp X$ by
\[
\p \dgpg X0 = \D(\barp \dgml X0).
\]
The $t$-structure associated in~\cite{a} to the given $s$-structure and to
a perversity $p$ is the pair $(\p\dgbl X0, \p\dgbg X0)$, where
\[
\p\dgbl X0 = \p\dgml X0 \cap \dgb X
\qquad\text{and}\qquad
\p\dgbg X0 = \p\dgpg X0 \cap \dgb X.
\]

The remainder of the section will be spent establishing a number of useful
lemmas about these objects.  Let $q: X^\gen \to \Z$ be a function such that
\begin{equation}\label{eqn:g-const}
q(x) = q(y)
\qquad\text{whenever}\qquad
\barGx = \barGy.
\end{equation}
Given such a function, let
\[
\q\cgl Xw = \left\{ \cF \in \cg X \,\Bigg|
\begin{array}{c}
\text{for any closed subscheme $i: \barGx \hto X$ with} \\
\text{$x \in X^\gen$, there is a dense open subscheme} \\
\text{$V \subset \barGx$ such that $i^*\cF|_V \in 
\cgl V{w+q(x)}$}
\end{array} \right\}.
\]
One may either regard this definition as a condition only on
\emph{reduced} closed subschemes of the form $\barGx$, or as a condition
on all possible closed subscheme structures on the various closed sets
$\barGx$.  These two interpretations are equivalent
by~\cite[Proposition~4.1]{a}, however, so there is no ambiguity in the
definition.  The first viewpoint is more convenient for checking explicit
examples, but the second is sometimes more useful in proofs. 

\begin{lem}\label{lem:qcgl-res}
Let $x \in X^\gen$, and let $i: Z \hto X$ be a closed subscheme structure
on $\barGx$.  For any sheaf $\cF \in \q\cgl Xw$ and any $r \ge 0$, there is
a dense open subscheme $V \subset Z$ such that $h^{-r}(Li^*\cF)|_V \in \cgl
V{w+q(x)}$.
\end{lem}
\begin{proof}
The proof of this lemma follows that of~\cite[Lemma~8.2]{a} nearly
verbatim.  By the definition of $\q\cgl Xw$, we know that there is a dense
open subset $Z'\subset Z$ such that $i^*\cF|_{Z'} \in \cgl {Z'}{w+q(x)}$. 
Let $X' = X \ssm (Z' \ssm Z)$.  Then $X'$ is a dense open subset of $X$,
and $i: Z' \hto X$ is a closed subscheme of $X'$.  It clearly suffices to
prove the lemma in the case where $X$ and $Z$ are replaced by $X'$ and
$Z'$.  We therefore henceforth assume, without loss of generality, that
$i^*\cF \in \cgl Z{w+q(x)}$.

We now proceed by induction on $r$.  For $r = 0$, the lemma is trivial: we
have $i^*\cF \in \cgl Z{w+q(x)}$ by assumption.  Now, suppose $r > 0$. 
According to Axiom~(S10) in the definition of an $s$-structure~\cite{a},
there is an open subscheme $V' \subset Z$ such that for any open set $U
\subset X$ with $U \cap Z \subset V'$, we have $\Ext^r(\cF|_U, i_*\cG|_U) =
0$ for all $\cG \in \cgg Z{w+q(x)+1}$.  (In fact, Axiom~(S10) guarantees
this vanishing for all $\cG$ in a slightly larger category, denoted $\tcgg
Z{w+q(x)+1}$, but we will not require that additional information.) 
Equivalently, for any open $V \subset V'$, we have $\Hom(Li^*\cF|_V,
\cG[r]|_V) = 0$ for all $\cG \in \cgg Z{w+q(x)+1}$.  We also have
$\Hom(Li^*\cF|_V, \cG[r]|_V) \cong \Hom(\tau^{\ge -r}Li^*\cF|_V,
\cG[r]|_V)$, and then from the distinguished triangle
\[
\tau^{\le -r}\tau^{\ge -r}Li^*\cF \to \tau^{\ge -r}Li^*\cF \to \tau^{\ge
-r+1}Li^*\cF \to
\]
we obtain the exact sequence
\begin{multline*}
\cdots \to \Hom(\tau^{\ge -r}Li^*\cF|_V, \cG[r]|_V) \to 
\Hom(\tau^{\le -r}\tau^{\ge -r}Li^*\cF|_V, \cG[r]|_V) \to \\
\Hom(\tau^{\ge -r+1}Li^*\cF[-1]|_V, \cG[r]|_V) \to \cdots.
\end{multline*}
Since $\tau^{\le -r}\tau^{\ge -r}Li^*\cF \cong h^{-r}(Li^*\cF)[r]$, the
sequence above can be rewritten as 
\begin{multline*}
\cdots \to \Hom(Li^*\cF|_V, \cG[r]|_V) \to 
\Hom(h^{-r}(Li^*\cF)|_V, \cG|_V) \to \\
\Hom(\tau^{\ge -(r-1)}Li^*\cF|_V, \cG[r+1]|_V) \to \cdots.
\end{multline*}
The first term above vanishes.  Note that 
\[
h^k(\tau^{\ge -(r-1)}Li^*\cF) \cong 
\begin{cases}
h^k(Li^*\cF) & \text{if $-(r-1) \le k \le 0$,} \\
0 & \text{otherwise.}
\end{cases}
\]
Thus, by the inductive assumption, the cohomology sheaves of $\tau^{\ge
-(r-1)}Li^*\cF$ have the property that for each $k$, there is a dense open
subscheme $V_k \subset Z$ such that $h^k(\tau^{\ge -(r-1)}Li^*\cF)|_{V_k}
\in \cgl {V_k}{w+q(x)}$.  This property is precisely the hypothesis
of~\cite[Lemma~8.1]{a}, which then tells us that there is a dense open
subscheme $V'' \subset Z$ such that the last term in the exact sequence
above vanishes whenever $V \subset V''$. In particular, let us take $V = V'
\cap V''$.  The middle term above then clearly vanishes.  Since
$\Hom(h^{-r}(Li^*\cF)|_V, \cG_1) = 0$ for all $\cG_1 \in \cgl V{w+q(x)+1}$,
we have $h^{-r}(Li^*\cF)|_V \in \cgl V{w+q(x)}$, as desired.
\end{proof}

\begin{lem}\label{lem:qcgl-serre}
$\q\cgl Xw$ is a Serre subcategory of $\cg X$.
\end{lem}
\begin{proof}
Suppose we have a short exact sequence $0 \to \cF' \to \cF \to \cF'' \to 0$
in $\cg X$.  Given $x \in X^\gen$ and a closed subscheme structure $i: Z
\hto X$ on $\barGx$, consider the exact sequence
\[
h^{-1}(Li^*\cF'') \to i^*\cF' \to i^*\cF \to i^*\cF'' \to 0.
\]
Suppose $\cF'$ and $\cF''$ are in $\q\cgl Xw$.  Then there are dense open
subschemes $V', V'' \subset Z$ such that $i^*\cF'|_{V'} \in \cgl
{V'}{w+q(x)}$ and $i^*\cF''|_{V''} \in \cgl {V''}{w+q(x)}$.  Let $V = V'
\cap V''$.  Then, since $\cgl V{w+q(x)}$ is a Serre subcategory of $\cg V$,
we see that $i^*\cF|_V \in \cgl V{w+q(x)}$, so $\cF \in \q\cgl Xw$.

Conversely, if $\cF \in \q\cgl Xw$, then there is a dense open subscheme $V
\subset Z$ such that $i^*\cF|_V \in \cgl V{w+q(x)}$.  It follows that
$i^*\cF''|_V \in \cgl V{w+q(x)}$ as well, so $\cF'' \in \q\cgl Xw$.  Next,
by Lemma~\ref{lem:qcgl-res}, there is some dense open subscheme $V' \subset
Z$ such that $h^{-1}(Li^*\cF'')|_{V'} \in \cgl {V'}{w+q(x)}$, and it
follows that $i^*\cF'|_{V \cap V'} \in \cgl{V \cap V'}{w+q(x)}$.  Thus,
$\cF' \in \q\cgl Xw$ as well.
\end{proof}

Next, let $p$ be a staggered perversity function.  The following alternate
characterization of $\p\dgml X0$ will be useful.

\begin{lem}\label{lem:dgml}
We have
\[
\p\dgml X0 = \{ \cF \in \dgm X \mid \text{$h^k(\cF) \in \pl\cgl
X{-k}$ for all $k \in \Z$} \}.
\]
\end{lem}

\begin{rmk}
Note the similarity between the right-hand side of this equation and the
definition of $\sD^{\leq 0}$ of definition \ref{defn:stag}.
\end{rmk}

\begin{proof}
Throughout the proof, $x$ will denote a point of $X^\gen$, and $i: Z \hto
X$ will denote a closed subscheme structure on $\barGx$.

First, suppose $\cF$ is concentrated in a single degree with respect to the
standard $t$-structure, say in degree $n$, and that $h^n(\cF) \in \pl\cgl
X{-n}$.  If $k > n$, then of course $h^k(Li^*\cF) = 0$.  If $k \le n$, then
by Lemma~\ref{lem:qcgl-res}, there is a dense open subscheme $V \subset Z$
such that $h^k(Li^*\cF)|_V \in \cgl X{p(x)-n} \subset \cgl X{p(x)-k}$, so
$\cF \in \p\dgml X0$.

Next, if $\cF \in \dgb X$ and $h^k(\cF) \in \pl\cgl X{-k}$ for all $k$, it
follows that $\cF \in \p\dgml X0$ by the preceding paragraph and a standard
induction argument on the number of nonzero cohomology sheaves of $\cF$. 
Finally, suppose that $\cF \in \dgm X$ and that $h^k(\cF) \in \pl\cgl
X{-k}$ for all $k$.  For any $k \in \Z$, $\tau^{\ge k}\cF$
is in $\dgb X$, so we already know that $\tau^{\ge k}\cF \in \p\dgml X0$. 
But consideration of the distinguished triangle
\[
Li^*\tau^{\le k-1}\cF \to Li^*\cF \to Li^*\tau^{\ge k}\cF \to
\]
shows that $h^k(Li^*\cF) \cong h^k(Li^*\tau^{\ge k}\cF)$, so in particular,
there is a dense open subscheme $V \subset Z$ with $h^k(Li^*\cF)|_V \in
\cgl V{p(x)-k}$, so $\cF \in \p\dgml X0$, as desired.

Conversely, suppose $\cF \in \p\dgml X0$.  Let $a$ be the largest integer
such that $h^a(\cF) \ne 0$.  Then of course $h^a(Li^*\cF) \cong
h^a(Li^*\tau^{\ge a}\cF) \cong i^*h^a(\cF)$, and we know that there is a
dense open subscheme $V \subset Z$ such that $i^*h^a(\cF)|_V \in \cgl
V{p(x)-a}$, so $h^a(\cF) \in \pl\cgl X{-a}$.  

Now, we will prove by downward induction on $k$ that $h^k(\cF) \in \pl\cgl
X{-k}$ and that $\tau^{\le k-1}\cF \in \p\dgml X0$ for all $k$.  These
statements hold trivially if $k > a$.  Suppose we know that $h^{k+1}(\cF)
\in \pl\cgl X{-k-1}$ and $\tau^{\le k}\cF \in \p\dgml X0$.  By the
preceding paragraph, we know that $h^k(\cF) = h^k(\tau^{\le k}\cF) \in
\pl\cgl X{-k}$.  Next, from the distinguished triangle $\tau^{\le k-1}\cF
\to \tau^{\le k}\cF \to \tau^{[k,k]}\cF \to$, we obtain the exact sequence
\[
h^{r-1}(Li^*\tau^{[k,k]}\cF) \to h^r(Li^*\tau^{\le k-1}\cF) \to
h^r(Li^*\tau^{\le k}\cF).
\]
Assume $r \le k-1$ (otherwise, the middle term above vanishes).  By
Lemma~\ref{lem:qcgl-res}, for some dense open $V \subset Z$,
$h^{r-1}(Li^*\tau^{[k,k]}\cF)|_V \in \cgl V{p(x)-k} \subset \cgl
V{p(x)-r}$.  Replacing $V$ by a smaller open subscheme if necessary, we may
also assume that $h^r(Li^*\cF)|_V \in \cgl V{p(x)-r}$.  It follows that
$h^r(Li^*\tau^{\le k-1}\cF)|_V \in \cgl V{p(x)-r}$, so $\tau^{\le k-1}\cF
\in \p\dgml X0$.

In particular, $h^k(\cF) \in \pl\cgl X{-k}$ for all $k$, as desired.
\end{proof}

In the course of the preceding proof, we have also established the
following statement.

\begin{cor}\label{cor:stag-trunc}
The category $\p\dgml X0$ is stable under all standard truncation functions
$\tau^{\le k}$ and $\tau^{\ge k}$.\qed
\end{cor}

\section{Baric Structures on Coherent Sheaves, II}
\label{sect:baric-coh2}

In this section, we achieve the main goal of the paper: the construction of
a class of baric structures on derived categories of equivariant coherent
sheaves.  These baric structures depend on a function on $X^\gen$ that
plays a role analogous to that played by a staggered perversity in
Section~\ref{sect:stagt}.

\begin{defn}\label{defn:rec}
Suppose $G$ acts on $X$ with finitely many orbits.  For each orbit $C
\subset X$, let $\cI_C \subset \cO_X$ denote the ideal sheaf corresponding
to the reduced closed subscheme structure on $\overline C \subset X$.  An
$s$-structure on $X$ is said to be \emph{recessed} if for each $C$,
$\cI_C/\cI_C^2 \in \cgl X{-1}$.
\end{defn}

For the remainder of the paper, we assume that $G$ acts on $X$ with
finitely many orbits, and that $X$ is endowed with a recessed
$s$-structure.  (See Remarks~\ref{rmk:ineq} and~\ref{rmk:ineq2}, however.) 
The assumption that
the $s$-structure is recessed is a mild one:  ``most'' of the
$s$-structures appearing in~\cite{t} are recessed, as is the one used
in~\cite{as}.  

Note that $\cI_C/\cI_C^2$ is always at least in $\cgl X0$, since it is a
subquotient of $\cO_X \in \cgl X0$.  In addition, since the coherent
pullback functor to a locally closed subscheme is right $s$-exact, it
follows that the restriction of a recessed $s$-structure to any locally
closed subscheme is also recessed.

\begin{rmk}
It is certainly possible to define the notion of ``recessed $s$-structure''
in a way that does not assume finiteness of the number of orbits.  (One
simply imposes a condition on the ideal sheaf of $\barGx$ for every $x \in
X^\gen$, not just for every orbit closure.)  However, it seems likely that
when there are infinitely many orbits, there are no recessed
$s$-structures.
\end{rmk}

Given a function $q: X^\gen \to \Z$ satisfying~\eqref{eqn:g-const}, define a a new function $\dualq:
X^\gen \to \Z$ given by
\[
\dualq(x) = \alt \barGx - q(x).
\]
Note that when $G$ acts on $X$ with finitely many orbits, a function $q: X^\gen \to \Z$ satisfying~\eqref{eqn:g-const} may be regarded as a $\Z$-valued function on the set of orbits.  It will sometimes be convenient to adopt this point of view, and, given an orbit $C \subset X$, we sometimes write
\[
q(C) = q(x_C)
\qquad\text{where $x_C \in X^\gen$ is any generic point of $C$}.
\]

\begin{lem}\label{lem:q-ext}
Let $\cG \in \cg X$, and let $j: U \hto X$ be an open subscheme.  Suppose
$\cF_1 \subset \cG|_U$ is such that $\cF_1 \in \q\cgl Uw$.  Then there
exists a subsheaf $\cF \subset \cG$ such that $\cF|_U \cong \cF_1$ and $\cF
\in \q\cgl Xw$.
\end{lem}
\begin{proof}
If $U$ is closed ({\it i.e.}, if $U$ is a connected component of $X$), then
$j_*\cF_1$ is naturally a subsheaf of $\cG$, so we simply take $\cF \cong
j_*\cF_1$.   Otherwise, let $C$ be an open orbit in $\overline U \ssm U$,
and let $V$ be the open subscheme $U \cup C$.  By induction on the number
of orbits in $\overline U \ssm U$, it suffices to find $\cF \subset \cG|_V$
such that $\cF \in \q\cgl Vw$ and $\cF|_U \cong \cF$.  Let $\kappa: C \hto
V$ be the inclusion map, and let $\cI_C$ be the ideal sheaf of $C$ in $V$. 
Finally, let $\cF'$ be some subsheaf of $\cG|_V$ such that $\cF'|_U \cong
\cF_1$.  Suppose $\kappa^*\cF' \in \cgl Cv$.  If $v \le w+q(C)$, we may
take $\cF = \cF'$, and we are finished.  On the other hand, if $v >
w+q(C)$, let $\cF = \cI_C^{v-w-q(C)}\cF'$.  Since $\cI_C|_U \cong \cO_U$,
we clearly still have $\cF|_U \cong \cF_1$.  The fact that the
$s$-structure is recessed means that $\kappa^*\cI_C \in \cgl C{-1}$, so
$\kappa^*\cI_C^{\otimes v-w-q(C)} \in \cgl C{-v+w+q(C)}$, and therefore
$\kappa^*\cI_C^{\otimes v-w-q(C)} \otimes \kappa^*\cF' \in \cgl C{w+q(C)}$.
 Now, $\kappa^*\cF$ is a quotient of $\kappa^*\cI_C^{\otimes v-w-q(C)}
\otimes \kappa^*\cF'$, so $\kappa^*\cF \in \cgl C{w+q(C)}$, as desired.
\end{proof}

Given a function $q: X^\gen \to \Z$, we define a full subcategory of $\dgm
X$ by
\[
\q\dsml Xw = \{ \cF \in \dgm X \mid \text{$h^k(\cF) \in \q\cgl Xw$} \}.
\]
We also define a full subcategory of $\dgp X$ by
\[
\q\dspg Xw = \D(\cheq\dsml X{-w}).
\]
Finally, we put
\[
\q\dsl Xw = \q\dsml Xw \cap \dgb X
\quad\text{and}\quad
\q\dsg Xw = \q\dspg Xw \cap \dgb X.
\]
The main result of the paper is the following.

\begin{thm}\label{thm:main}
The collection of subcategories $(\{\q\dsl Xw\}, \{\q\dsg Xw\})_{w \in \Z}$
is a bounded, nondegenerate HLR baric structure on $X$.
\end{thm}

The proof of this theorem will occupy the rest of this section.  Note that
the definition of $\q\dsml Xw$ is consistent with the notation used in
Section~\ref{sect:baric-coh1}.  We will see in Corollary~\ref{cor:qdspg}
that the same holds for $\q\dspg Xw$.

\begin{lem}\label{lem:subcat}
$\q\dsl Xw$ and $\q\dsg Xw$ are thick subcategories of $\dgb X$.  Moreover,
$\q\dsl Xw \subset \q\dsl X{w+1}$, and $\q\dsg Xw \supset \q\dsg X{w+1}$.
\end{lem}
\begin{proof}
It is obvious that $\q\dsl Xw$ is stable under shift.  Since it is defined
by the requirement that cohomology sheaves belong to a Serre subcategory of
$\cg X$ (see Lemma~\ref{lem:qcgl-serre}), it is stable under extensions as
well, so it is indeed a thick subcategory of $\dgb X$.  It follows that
$\q\dsg Xw$ is as well.  It is obvious that $\q\dsl Xw \subset \q\dsl
X{w+1}$, and hence that $\q\dsg Xw \supset \q\dsg X{w+1}$.
\end{proof}

\begin{lem}\label{lem:baric-res}
Let $j: U \hto X$ be the inclusion of an open subscheme, and $i: Z \hto X$
the inclusion of a closed subscheme.  Then:
\begin{enumerate}
\item $j^*$ takes $\q\dsml Xw$ to $\q\dsml Uw$ and $\q\dspg Xw$ to
$\q\dspg Uw$.
\item $Li^*$ takes $\q\dsml Xw$ to $\q\dsml Zw$.
\item $Ri^!$ takes $\q\dspg Xw$ to $\q\dspg Zw$.
\item $i_*$ takes $\q\dsml Zw$ to $\q\dsml Xw$ and $\q\dspg Zw$ to
$\q\dspg Xw$.
\end{enumerate}
\end{lem}

This statement closely resembles Lemma~\ref{lem:hlr-baric-res}; indeed, it
would merely be an instance of that lemma if Theorem~\ref{thm:main} were
already known.  However, the proof of Theorem~\ref{thm:main} depends on
this lemma, so we must give it an independent proof.

\begin{proof}
(1)~It is immediate from the definition of $\q\cgl Xw$ that $j^*$ takes
$\q\cgl Xw$ to $\q\cgl Uw$.  Since $j^*$ is an exact functor, it follows
that it takes $\q\dsml Xw$ to $\q\dsml Uw$.  Since $j^*$ commutes with
$\D$, we also see that it takes $\q\dspg Xw$ to $\q\dspg Uw$.

(2)~We proceed by noetherian induction: assume the statement is known if
$X$ is replaced by a proper closed subscheme, or if $X$ is retained and $Z$
is replaced by a proper closed subscheme.  Suppose $\cF \in \q\dsml Xw$. 
We show by downward induction on $k$ that $h^k(Li^*\cF) \in \q\cgl Zw$. 
For large $k$, $h^k(Li^*R\cF) = 0$, so this holds trivially.  Now, assume
that $h^r(Li^*\cF) \in \q\cgl Zw$ for all $r > k$, and consider the
distinguished triangle $\tau^{\le k}Li^*\cF \to Li^*\cF \to \tau^{\ge
k+1}Li^*\cF \to$.  Then $\tau^{\ge k+1}Li^*\cF$ is an object of $\q\dsl
Zw$, so for any $x \in Z^\gen$ and any closed subscheme structure $\kappa:
Y \hto Z$ on $\barGx$, we know that $L\kappa^*\tau^{\ge k+1}Li^*\cF \in
\q\dsml Yw$.  Consider the exact sequence
\[
h^{k-1}(L\kappa^*\tau^{\ge k+1}Li^*\cF) \to h^k(L\kappa^*\tau^{\le
k}Li^*\cF) \to h^k(L\kappa^*Li^*\cF).
\]
The first term above belongs to $\q\cgl Yw$.  Observe that
$h^k(L\kappa^*\tau^{\le k}Li^*\cF) \cong \kappa^*h^k(Li^*\cF)$.  Thus, to
prove that $h^k(Li^*\cF) \in \q\cgl Zw$, we must show that there is a
dense open subscheme $V \subset Y$ such that $h^k(L\kappa^*\tau^{\le
k}Li^*\cF)|_V \in \cgl V{w+q(x)}$.

If $Y$ is a proper closed subscheme of $Z$, then we have assumed
inductively that $L(\kappa \circ i)^*\cF \in \q\dsml Yw$, and in that case,
the last term in the sequence above belongs to $\q\cgl Yw$ as well.  By
Lemma~\ref{lem:qcgl-serre}, the middle term as well, and the existence of
the desired open subscheme $V \subset Y$ follows.

On the other hand, if $Y = Z$, and $\kappa$ is the identity map, then
Lemma~\ref{lem:qcgl-res} gives us a dense open subscheme $V' \subset Z$
such that $h^k(Li^*\cF)|_{V'} \in \cgl {V'}{w+q(x)}$.  The fact that
$h^{k-1}(\tau^{\ge k+1}Li^*\cF) \in \q\cgl Zw$ implies that there a dense
open subscheme $V'' \subset Z$ with $h^{k-1}(\tau^{\ge k+1}Li^*\cF)|_{V''}
\in \q\cgl {V''}{w+q(x)}$.  If we let $V = V' \cap V''$, then we see from
the exact sequence above that $h^k(\tau^{\le k}Li^*\cF)|_V \in \cgl
V{w+q(x)}$, as desired.

(3)~If $\cF \in \q\dspg Xw$, let $\cF' \in \cheq\dsml X{-w}$ be such that
$\D\cF' \cong \cF$.  Then $Ri^! \cF \cong \D(Li^*\cF') \in \q\dspg Zw$
since $Li^*\cF' \in \cheq\dsml Zw$.

(4)~Since $\q\dsml Zw$ and $\q\dsml Xw$ are defined by conditions on their
cohomology sheaves, the first statement follows from the fact that $i_*$ is
an exact functor taking $\q\cgl Zw$ to $\q\cgl Xw$.  The second statement
follows by duality.
\end{proof}

\begin{prop}\label{prop:baric-orth}
If $\cF \in \q\dsml Xw$ and $\cG \in \q\dspg X{w+1}$, then $\Hom(\cF,\cG) =
0$.
\end{prop}
\begin{proof}
We proceed by noetherian induction: assume the theorem is known for all
proper closed subschemes of $X$.  Let $a$ and $b$ be such that $\cG \in
\dgpg Xa$ and $\cF \in \dgml Xb$.  Since $\Hom(\cF,\cG) \cong
\Hom(\tau^{\ge a}\cF, \cG)$, we may replace $\cF$ by $\tau^{\ge a}\cF$ and
assume that $\cF \in \q\dsl Xw$.  Next, let $\cG' \in \dsml X{-w-1}$ be
such that $\D\cG' \cong \cG$.  For a  sufficiently small integer $c$, we
will have $\D(\tau^{\le c}\cG') \in \dgpg X{b+1}$.  From this, it follows
that $\Hom(\cF,\cG) \cong \Hom(\cF,\D(\tau^{\ge c+1}\cG'))$.  Replacing
$\cG$ by $\D(\tau^{\ge c+1}\cG')$, we may assume that $\cG \in \q\dsg Xw$.

With $\cF$ and $\cG$ both in $\dgb X$, induction on the number of
cohomology sheaves allows us to reduce to the case where both $\cF$ and
$\cG' := \D\cG$ are concentrated in a single degree.  By shifting both
objects simultaneously, we may assume without loss of generality that
$\cF \in \cg X$.  Let $x$ be a generic
point of $X$.  There is an open subscheme $U \subset X$ containing $x$ such
that $\cG'|_U \in \cgl U{\alt \barGx - q(x)-w-1}$.  By~\cite[Remark~3.2 and
Lemmas~6.1--6.2]{a}, we may replace $U$ by a smaller open subscheme
containing $x$ such that $\cG|_U$ is concentrated in a single degree, say
$d$, and such that $\cG[d]|_U \in \cgg U{q(x)+w+1}$.  If $d > 0$, then
clearly $\Hom(\cF|_U, \cG|_U) = 0$.  Otherwise, we
invoke~\cite[Axiom~(S9)]{a} to replace $U$ by a smaller open subscheme such
that $\Hom(\cF|_U, \cG|_U) = 0$.  Let $Z$ be the complementary closed
subspace to $U$, and consider the exact sequence 
\[
\lim_{\substack{\to \\ Z'}} \Hom(Li^*_{Z'}\cF, Ri^!_{Z'}\cG) \to
\Hom(\cF,\cG) \to \Hom(\cF|_U,\cG|_U),
\]
where $i_{Z'}: Z' \hto X$ ranges over all closed subscheme structures on
$Z$.  We have just seen that the last term vanishes.  Since $Li^*_{Z'}\cF
\in \q\dsml {Z'}w$ and $Ri^!_{Z'}\cG \in \q\dspg {Z'}{w+1}$, the first term
vanishes by induction.  So $\Hom(\cF,\cG) = 0$, as desired.
\end{proof}

\begin{prop}\label{prop:baric-dt}
For any $\cF \in \dgb X$, there is a distinguished triangle $\cF' \to \cF
\to \cF'' \to$ with $\cF' \in \q\dsl Xw$ and $\cF'' \in \q\dsg X{w+1}$. 
Moreover,
if $\cF \in \dgbg X0$, then $\cF'$ and $\cF''$ lie in $\dgbg X0$ as well.
\end{prop}
\begin{proof}
Once again, we proceed by noetherian induction, and assume the result is
known for all proper closed subschemes of $X$.  Now, assume first that
$\cF$ is a sheaf.  Let $C \subset X$ be an open (and possibly
nonreduced) orbit, and let $i:
\overline C \hto X$ be the inclusion of its closure.  By
Lemma~\ref{lem:q-ext}, there exists a subsheaf  $\cF_1 \subset \cF$ such
that $\cF_1 \in \q\cgl Xw$ and $\cF_1|_C \cong \sigma_{\le
w+q(C)}(\cF|_C)$.  Next, form a short exact sequence
\[
0 \to \cF_1 \to \cF \to \cG \to 0.
\]
Let $b = \cod \overline C$.  Then $i_*Ri^!\D\cG \in \dgpg Xb$, and,
by~\cite[Lemma~6.1]{a}, we know that $i_*Ri^!\D\cG|_C \cong \D\cG|_C$ is
concentrated in degree $b$.  Furthermore,~\cite[Proposition~6.8]{a} tells
us that $\D\cG[b]|_C \in \cgl C{\alt \overline C - q(C) -w - 1}$.  (If $C$ is
reduced, these assertions about $\D\cG|_C$ are immediate from the fact that
$\D$ is an exact functor, but in general, we must invoke~\cite[Lemma~6.1
and Proposition~6.8]{a}.)  Now, we use Lemma~\ref{lem:q-ext} again to find
a subsheaf $\cG_1 \subset h^b(i_*Ri^!\D\cG)$ such that $\cG_1 \in \cheq\cgl
X{-w-1}$ and $\cG_1|_C \cong \D\cG[b]|_C$.  Form the composition
\[
\cG_1[-b] \to h^b(i_*Ri^!\D\cG)[-b] \cong \tau^{\le b}i_*Ri^!\D\cG \to
i_*Ri^!\D\cG \to \D\cG,
\]
and then complete it to a distinguished triangle
\[
\cG_1[-b] \to \D\cG \to \cG' \to.
\]
Here, $\cG'$ is necessarily supported on the complement of $C$.  Let $\cF_2
= \D(\cG_1[-b])$, and let $\cH = \D\cG'$, so we have a distinguished
triangle
\[
\cH \to \cG \to \cF_2 \to.
\]
Since $\cod \overline C = b$, we see that $\cF_2 \in \dgbg X0$.  This
distinguished triangle then implies that $\cH \in \dgbg X0$ as well.  Note
also that $\cF_2 \in \q\dsg X{w+1}$, and that
\[
\cF \in \{\cF_1\} * \{\cH\} * \{\cF_2\}.
\]
Since $\cF_1 \in \q\dsl Xw$, $\cF_2 \in \q\dsg X{w+1}$, and $\cH$ is
supported on a proper closed subscheme, we conclude that $\cF \in \q\dsl Xw
* \q\dsg X{w+1}$, as desired.  The last statement of the proposition holds
by noetherian induction as well, since $\cF_1$, $\cH$, and $\cF_2$ all lie
in $\dgbg X0$ by construction.

The result also follows for any object of $\dgb X$ that is concentrated in
a single degree.  Finally, for general objects $\cF \in \dgb X$, we proceed
by induction on the number of nonzero cohomology sheaves.  Let $a \in \Z$
be such that $\tau^{\le a}\cF$ and $\tau^{\ge a+1}\cF$ are both nonzero. 
Then, they both have fewer nonzero cohomology sheaves than $\cF$, and we
assume inductively that there exist distinguished triangles
\begin{gather*}
\cF'_1 \to \tau^{\le a}\cF \to \cF''_1 \to, \\
\cF'_2 \to \tau^{\ge a+1}\cF \to \cF''_2 \to
\end{gather*}
with $\cF'_1, \cF'_2 \in \q\dsl Xw$ and $\cF''_1, \cF''_2 \in \q\dsg
X{w+1}$.  Consider the composition
\[
\cF'_2[-1] \to (\tau^{\ge a+1}\cF)[-1] \to \tau^{\le a}\cF \to \cF''_1.
\]
By Proposition~\ref{prop:baric-orth}, this composition is $0$, so we see
from the exact sequence
\[
\Hom(\cF'_2[-1], \cF'_1) \to \Hom(\cF'_2[-1], \tau^{\le a}\cF) \to
\Hom(\cF'_2[-1],\cF''_1)
\]
that the morphism $\cF'_2[-1] \to (\tau^{\ge a+1}\cF)[-1] \to \tau^{\le
a}\cF$ factors through $\cF'_1$.  That is, we have a commutative square
\[
\xymatrix@=10pt{
\cF'_2[-1] \ar[r]\ar[d] & (\tau^{\ge a+1}\cF)[-1] \ar[d] \\
\cF'_1 \ar[r] & \tau^{\le a}\cF}
\]
We define objects $\cF', \cF'' \in \dgb X$ by completing this diagram 
as follows, using the $9$-lemma~\cite[Proposition~1.1.11]{bbd}:
\[
\xymatrix@=10pt{
\cF'_2[-1] \ar[r]\ar[d] & (\tau^{\ge a+1}\cF)[-1] \ar[r]\ar[d] &
\cF''_2[-1] \ar[r]\ar[d] & {}\\
\cF'_1 \ar[r]\ar[d] & \tau^{\le a}\cF \ar[r]\ar[d] & \cF''_1 \ar[r]\ar[d]
&{}\\
\cF' \ar[r]\ar[d] & \cF \ar[r]\ar[d] & \cF'' \ar[r]\ar[d] & {} \\
&&&}
\]
Since $\q\dsl Xw$ and $\q\dsg X{w+1}$ are stable under shift and
extensions, we see that $\cF' \in \q\dsl Xw$ and $\cF'' \in \q\dsg X{w+1}$,
as desired.  Moreover, if $\cF$ lies in $\dgbg X0$, then so do $\tau^{\le
a}\cF$ and $\tau^{\ge a+1}\cF$, and hence, by induction, the objects
$\cF'_1$, $\cF''_1$, $\cF'_2$, and $\cF''_2$ all lie in $\dgbg X0$ as well.
 It then follows that $\cF'$ are $\cF''$ are in $\dgbg X0$, as desired.
\end{proof}

\begin{proof}[Proof of Theorem~\ref{thm:main}]
Lemma~\ref{lem:subcat} and Propositions~\ref{prop:baric-orth}
and~\ref{prop:baric-dt} together state that all the axioms for a baric
structure hold.  Moreover, the last part of Proposition~\ref{prop:baric-dt}
tells us that the baric truncation functors are left $t$-exact (with
respect to the standard $t$-structure), and it is obvious from the
definition of $\q\dsl Xw$ that it is preserved by the truncation functors
$\tau^{\le n}$ and $\tau^{\ge n}$.  Thus, the baric structure $(\{\q\dsl
Xw\}, \{\q\dsg Xw\})_{w\in \Z}$ is compatible with the standard
$t$-structure. Next,
for any closed subscheme $i: Z \hto X$, Lemma~\ref{lem:baric-res} tells us
that $Li^*$ is right baryexact and that $Ri^!$ is left baryexact.  Thus,
this baric structure is hereditary, and hence HLR by Theorem~\ref{thm:hlr}.

It remains to prove that the baric structure is bounded (and therefore
nondegenerate).  Every sheaf in $\cg X$ belongs to some $\cgl Xn$, and
hence to some $\q\cgl Xw$ (simply take $w$ to be the maximum value of $n -
q(x)$).  Since an object $\cF \in \dgb X$ has finitely many nonzero
cohomology sheaves, we can clearly find a $w$ such that all its cohomology
sheaves belong to $\q\cgl Xw$, so that $\cF \in \q\dsl Xw$.  The same
reasoning yields an integer $v$ such that $\D\cF \in \cheq\dsl X{-v}$, and
hence $\cF \in \q\dsg Xv$.  Thus, the baric structure is bounded and
nondegenerate.
\end{proof}

We can now verify that the notation $\q\dspg Xw$ is consistent with the
notation of Section~\ref{sect:baric-coh1}.

\begin{cor}\label{cor:qdspg}
We have
\[
\q\dspg Xw = \{ \cF \in \dgp X \mid \text{$\q\beta_{\le w-1}\tau^{\le k}\cF
\in \dgbg X{k+2}$ for all $k$} \}.
\]
\end{cor}
\begin{proof}
We have already observed that the definition of $\cheq\dsml X{-w}$ is
consistent with the notation of Section~\ref{sect:baric-coh1}, so by
Lemma~\ref{lem:unbdd-orth}, for $\cF \in \dgm X$, we have $\cF \in
\cheq\dsml X{-w}$ if and only if $\Hom(\cF,\cG) = 0$ for all $\cG \in
\cheq\dsg X{-w+1}$.  Applying $\D$, we have $\cF \in \q\dspg Xw$ if and
only if $\Hom(\D\cF,\D\cG) = 0$ for all $\cG \in \q\dsl X{w-1}$, or,
equivalently, if $\Hom(\cG,\cF) = 0$ for all $\cG \in \q\dsl X{w-1}$.  The
corollary follows by another application of Lemma~\ref{lem:unbdd-orth}.
\end{proof}

\begin{rmk}\label{rmk:ineq}
The proof of Lemma~\ref{lem:q-ext} depends in an essential way on the
assumption of finitely many orbits and a recessed $s$-structure, but no
other arguments
given in this section do.  (The role of the orbit closure $\overline C$ in
the proof of Proposition~\ref{prop:baric-dt} could instead have been played
by $\barGx$ for some generic point $x$.)  By imposing additional conditions
that permit us
to evade Lemma~\ref{lem:q-ext}, we can find a version of
Theorem~\ref{thm:main} that holds in much greater generality.

Specifically, assume that the function $q: X^\gen \to \Z$ is
\emph{monotone}: that is, if $x \in \barGy$, then $q(x) \ge q(y)$.  
Suppose we have a coherent sheaf $\cG \in \cg X$, an open subscheme $j: U
\hto X$, and a subsheaf $\cF_1 \subset \cG|_U$ with $\cF_1 \in \q\cgl Uw$. 
By replacing $U$ by a smaller open subscheme, we may assume that $\cF_1 \in
\cgl U{q(x)+w}$, where $x$ is a generic point of $U$.  Then $\cF_1$ is a
subsheaf of $\sigma_{\le q(x)+w}\cG|_U$, and standard arguments show that
there is a subsheaf $\cF \subset \sigma_{\le q(x)+w}\cG$ supported on
$\overline U$ such that $\cF|_U
\cong \cF_1$.  The monotonicity assumption then implies that $\cF \in
\q\cgl Xw$.  This reasoning can be substituted for
invocations of
Lemma~\ref{lem:q-ext} for $\q\cgl Xw$.  Similarly, if $q$ is
\emph{comonotone}, meaning that $\dualq$ is monotone, then the reasoning
above can replace invocations of Lemma~\ref{lem:q-ext} for the category
$\cheq\cgl Xw$.  The proof of Theorem~\ref{thm:main} uses
Lemma~\ref{lem:q-ext} in both these ways.

We thus obtain the following
result: suppose $X$ is a scheme satisfying the assumptions of
Section~\ref{sect:stagt}, equipped with an $s$-structure. In particular, we
do not assume that $G$ acts with finitely orbits, or that the $s$-structure
is recessed.  If $q: X^\gen \to \Z$ is both monotone and comonotone, then
the collection of subcategories $(\{\q\dsl Xw\}, \{\q\dsg Xw\})_{w \in \Z}$
is a bounded, nondegenerate HLR baric structure on $X$.
\end{rmk}

\section{Multiplicative Baric Structures and $s$-structures}
\label{sect:mult}

In this section we study the relationship between multiplicative baric
structures on the triangulated category $\dgb X$ and $s$-structures on the
abelian category $\cg X$.  The authors had originally hoped that under
appropriate conditions the two notions would be equivalent, and that the
developments in sections \ref{sect:stagt} and \ref{sect:baric-coh2} could
be simplified by replacing the latter concept with the former.  In other
words, the hope was that there would be a one-to-one correspondence between
multiplicative HLR baric structures and $s$-structures on a $G$-scheme $X$.

This turns out to be not quite correct.  Rather, we prove here that there is a one-to-one correspondence between multiplicative baric structures and a certain class of \emph{pre}-$s$-structures, including all $s$-structures.  (A pre-$s$-structure is a collection of subcategories of $\cg X$ satisfying the first six of the ten axioms for an $s$-structure in~\cite{a}.)  It would be interesting to look for an additional axiom on multiplicative baric structures that is satisfied precisely by those baric structures corresponding to $s$-structures, but we have not pursued this here.

We say that a baric structure $\schemebaric X$ is \emph{multiplicative} if
either of the following two equivalent conditions holds:
\begin{enumerate}
\item If $\cF \in \dsl Xw$ and $\cG \in \dsl Xv$, then $\cF \Lotimes \cG
\in \dsml X{w+v}$.
\item If $\cF \in \dsl Xw$ and $\cG \in \dsg Xv$, then $\cRHom(\cF,\cG)
\in \dspg X{v-w}$.
\end{enumerate}

\begin{thm}\label{thm:mult}
Suppose $\schemebaric X$ is a multiplicative baric structure on $X$.  Then
the categories
\begin{align*}
\cgl Xw &= \cg X \cap \dsl Xw, \\
\cgg Xw &= \{ \cF \in \cg X \mid \text{$\Hom(\cG,\cF) = 0$ for all $\cG
\in \cgl X{w-1}$} \}
\end{align*}
constitute a pre-$s$-structure on $X$.

Conversely, given an $s$-structure $(\{\cgl Xw\}, \{\cgg Xw\})_{w\in\Z}$ on
a
scheme $X$ with finitely many $G$-orbits, the categories
\begin{align*}
\dsl Xw &= \{ \cF \in \dgb X \mid \text{$h^k(\cF) \in \cgl Xw$ for all $k
\in \Z$} \}, \\
\dsg Xw &= \{ \cF \in \dgb X \mid \text{$\Hom(\cG,\cF) = 0$ for all $\cG
\in \dsl X{w-1}$} \}
\end{align*}
constitute a multiplicative baric structure on $X$.
\end{thm}
\begin{proof}
Suppose first that $\schemebaric X$ is a multiplicative baric structure on
$X$.  To show that the categories above constitute a pre-$s$-structure, we
must verify axioms (S1)--(S6) from~\cite{a}.  (The reader is referred
to~\cite{a} for the statements of these axioms.)

Axioms (S2) and (S3) are
clear from the definitions, and axiom (S1) follows from the fact that
$\schemebaric X$ is compatible with the standard $t$-structure.  

Let us prove axiom (S4).  Let $\cF$ be an object of $\cg X$.  Since $\cF$
is noetherian, and
$\cgl Xw$ is a Serre subcategory, there is a largest subobject $\cF'
\subset \cF$ belonging 
to $\cgl Xw$.  Then $\cF/\cF'$ must belong to $\cgg X{w+1}$: otherwise,
there is a nonzero
map $\cG \to \cF/\cF'$ whose image $\cI \neq 0$ belongs to $\cgl Xw$, but
the inverse image of $\cI$ in $\cF$ contains the maximal $\cF'$.

Axiom (S5) follows from the fact that the baric structure on $\dgb X$ is
bounded, and Axiom (S6) follows from the multiplicativity of the baric
structure and the fact that for $\cF, \cG \in \cg X$, we have $\cF \otimes
\cG \cong h^0(\cF \Lotimes \cG)$.

Now, suppose we are given an $s$-structure $(\{\cgl Xw\}, \{\cgg
Xw\})_{w \in \Z}$. 
Let $\bz$ denote the constant function $X^\gen \to \Z$ of value $0$.
We claim that $\lbz\cgl Xw = \cgl Xw$.  It is clear from the definition
that $\cgl Xw \subset \lbz \cgl Xw$.  Conversely, if $x \in X^\gen$ is a
generic point of the support of an object $\cF \notin \cgl Xw$, it follows
from the gluing theorem for $s$-structures~\cite[Theorem~5.3]{a} that
there is no open subscheme $V \subset \barGx$ such that the restriction of
$\cF$ to $V$ lies in $\cgl Vw$, so $\cF \notin \lbz \cgl Xw$.  Since
$\lbz\cgl Xw = \cgl Xw$, we see that the categories $\schemebaric X$
defined in the statement of the theorem coincide with the baric structure
constructed in Theorem~\ref{thm:main} by taking $q = \bz$.  The fact that
this baric structure is multiplicative is a consequence of
Proposition~\ref{prop:baric-mult} below.
\end{proof}

\begin{prop}\label{prop:baric-mult}
Let $X$ be a scheme with finitely many $G$-orbits, and let $p,q: X^\gen
\to \Z$ be functions satisfying~\eqref{eqn:g-const}.  Suppose $\cF \in \pl\dsml Xw$.
\begin{enumerate}
\item If $\cG \in \q\dsml Xv$, then $\cF \Lotimes
\cG \in {}_{p+q}\dsml X{w+v}$.
\item If $\cG \in \q\dspg Xv$, then
$\cRHom(\cF,\cG) \in {}_{\hat p + q}\dspg X{v-w}$.
\end{enumerate}
\end{prop}
\begin{proof}
(1)~We will show by noetherian induction that $\Hom(\cF \Lotimes \cG, \cH)
= 0$ for all $\cH \in {}_{p+q}\dsg X{w+v+1}$.  Assume the result is known
for all proper closed subschemes of $X$, and let $C \subset X$ be an open
orbit.  Let $Z$ denote the closed subset $X \ssm C$, and consider the exact
sequence
\[
\lim_{\substack{\to \\ Z'}} \Hom(Li^*(\cF \Lotimes \cG),
Ri^!\cH) \to
\Hom(\cF \Lotimes \cG, \cH) \to \Hom((\cF \Lotimes \cG)|_C,\cH|_C),
\]
where $i': Z' \hto X$ ranges over all closed subscheme structures on $Z$. 
Now, $Li^*(\cF \Lotimes \cG) \cong Li^*\cF \Lotimes Li^*\cG$.  We have
$Li^*\cF \in \pl\dsml {Z'}w$ and $Li^*\cG \in \q\dsml {Z'}v$ by
Lemma~\ref{lem:baric-res}, and then $Li^*\cF \Lotimes Li^*\cG \in
{}_{p+q}\dsml {Z'}{w+v}$ by assumption.  We also have $Ri^!\cH \in
{}_{p+q}\dspg {Z'}{w+v+1}$, so the first term above clearly vanishes.

It now suffices to show that $(\cF \Lotimes \cG)|_C \in {}_{p+q}\dsml
C{w+v}$: that implies the vanishing of the last term in the exact sequence
above, and hence of the middle term as well.  Recall that on a
single $G$-orbit, the tensor product functor is exact (because all objects
of $\cg C$ are locally free), so there is a natural isomorphism
\[
h^r((\cF \Lotimes \cG)|_C) \cong \bigoplus_{i+j = r} h^i(\cF|_C) \otimes
h^j(\cG|_C).
\]
We know that $h^i(\cF|_C) \in \cgl C{w+p(C)}$ for all $i$, and that
$h^j(\cG|_C) \in \cgl C{v+q(C)}$ for all $j$.  It follows that $h^r((\cF
\Lotimes \cG)|_C) \in \cgl C{w+v+p(C)+q(C)}$ for all $r$, and hence that
$(\cF \Lotimes \cG)|_C \in {}_{p+q}\dsml C{w+v}$, as desired.

(2)~Consider $\D\cG \in \cheq\dsml X{-v}$.  By part~(1), $\cF
\Lotimes \D\cG \in {}_{p + \dualq}\dsml X{w-v}$.  Since $\cRHom(\cF,\cG)
\cong \D(\cF \Lotimes \D\cG)$, the result follows.
\end{proof}

\section{Staggered Sheaves}
\label{sect:stag2}

In this section, we retain the assumptions that $G$ acts on $X$ with
finitely many orbits, and that $X$ is equipped with a recessed
$s$-structure.

Given a function $q: X^\gen \to \Z$, let us define full subcategories of
$\dgm X$ and $\dgp X$ as follows:
\begin{align*}
\uq\dgml X0 &= \{ \cF \in \dgm X \mid \text{$h^k(\cF) \in \q\dsl X{-k}$ for
all $k \in \Z$} \}, \\
\uq\dgpg X0 &= \{ \cG \in \dgp X \mid \text{$\Hom(\cF[1],\cG) = 0$ for all
$\cF \in \uq\dgml X0$}\}.
\end{align*}
We also define bounded versions of these categories:
\[
\uq\dgbl X0 = \uq\dgml X0 \cap \dgb X
\quad\text{and}\quad
\uq\dgbg X0 = \uq\dgpg X0 \cap \dgb X.
\]
Let $\cG \in \uq\dgpg X0$.  There is some integer $n$ such that $\cG \in
\dgpg Xn$, and then for any $\cF \in \dgm X$, we have $\Hom(\cF[1],\cG)
\cong \Hom(\tau^{\ge n}(\cF[1]),\cG)$, with $\tau^{\ge n}(\cF[1]) \in \dgb
X$.  Thus, the definition of $\uq\dgpg X0$ could be changed to require
$\Hom(\cF[1],\cG)$ to vanish only when $\cF \in \uq\dgbl X0$.  By
Proposition~\ref{prop:compat}\eqref{it:ssorth}, it follows that
\[
\uq\dgbg X0 = \{ \cG \in \dgb X \mid \text{$\q\beta_{\le k}\cG \in \dgbg
X{-k}$ for all $k \in \Z$} \}.
\]
The categories $\uq\dgbl X0$ and $\uq\dgbg X0$ are none other than the
categories associated in Definition~\ref{defn:stag} to the standard
$t$-structure on $\dgb X$ with the baric structure $(\{\q\dsl Xw\},
\{\q\dsg Xw\})_{w\in \Z}$.

\begin{thm}\label{thm:stag-coh}
The categories $(\uq\dgbl X0,\uq\dgbg X0)$ constitute a bounded,
nondegenerate $t$-structure on $\dgb X$.
\end{thm}

\begin{defn}
The $t$-structure $(\uq\dgbl X0, \uq\dgbg X0)$ is called the
\emph{staggered $t$-structure of perversity $q$.}  Its heart, denoted
$\uq\cM(X)$, is the category of \emph{staggered sheaves}.
\end{defn}

This terminology and notation is consistent with that of~\cite{a} by
Lemma~\ref{lem:dgml}.  That is, if $q$ happens to be a perversity function
in the sense of~\cite{a}, then the $t$-structure constructed here coincides
with the $t$-structure associated to $q$ in~\cite{a}.  However, neither
this theorem nor the main result of~\cite{a} encompasses the other:
in~\cite{a}, no assumptions were made on the number of orbits or the
$s$-structure; here, no restrictions are imposed on the function $q: X^\gen
\to \Z$.

Note that if $q$ happens to be a perversity function in the sense
of~\cite{a}, then Theorem~\ref{thm:stag-coh} follows immediately from
Lemma~\ref{lem:dgml}, but it general, Theorem~\ref{thm:stag-coh} produces
$t$-structures that are not given by the construction of~\cite{a}.

\begin{proof}[Proof of Theorem~\ref{thm:stag-coh}]
We will prove this theorem by invoking Theorem~\ref{thm:stag-gen}.  To that
end, we must define an invariant $\mu(\cF)$ for any object $\cF \in \dgb X$
satisfying the hypotheses of that theorem.  For any nonzero object $\cF \in
\dgb X$, let
\[
m(\cF) = \min \{ k \in \Z \mid h^k(\cF) \ne 0 \}.
\]
Let $C$ be the maximum value of $\cod Z$ as $Z$ ranges over all closed
subschemes of $X$.  (Of course, $\cod Z$ takes only finitely many distinct
values, since $X$ has finite Krull dimension.)  Note that $\D(\dgbg X0)
\subset \dgbl XC$, and, more generally, $\D(\dgbg Xn) \subset \dgbl
X{C-n}$.  Let
\[
\mu(\cF) =
\begin{cases}
C + 1 - m(\cF) - m(\D\cF) & \text{if $\cF \ne 0$,} \\
0 & \text{if $\cF = 0$.}
\end{cases}
\]
We first prove that $\mu(\cF) > 0$ whenever $\cF \ne 0$.  If $m(\cF) = n$,
then $\cF \in \dgbg Xn$, so $\D\cF \in \dgbl X{C-n}$, and in particular,
$m(\D\cF) \le C-n$.  It follows that
\[
\mu(\cF) = C + 1 - m(\cF) - m(\D\cF) \ge C + 1 - n - (C-n) = 1,
\]
as desired.

Next, the left $t$-exactness of $\q\beta_{\le -n}$ implies that if $m(\cF)
= n$, then $\q\beta_{\le -n}\cF \in \dgbg Xn$, so $m(\q\beta_{\le -n}\cF)
\ge n$.  Now, consider the distinguished triangle
\[
\D\q\beta_{\ge -n+1}\cF \to \D\cF \to \D\q\beta_{\le -n}\cF \to.
\]
Since $\D\q\beta_{\ge -n+1}\cF \in \cheq\dsl X{n-1}$ and $\D\q\beta_{\le
-n}\cF \in \cheq\dsg Xn$, we have canonical isomorphisms
\[
\D\q\beta_{\ge -n+1}\cF \cong \cheq\beta_{\le n-1}\D\cF
\qquad\text{and}\qquad
\D\q\beta_{\le -n}\cF \cong \cheq\beta_{\ge n}\D\cF.
\]
Now, the left $t$-exactness of $\cheq\beta_{\ge n}$ shows that
$m(\D\q\beta_{\le -n}\cF) \ge m(\D\cF)$.

Finally, consider $\tau^{\ge n+1}\q\beta_{\le -n}\cF$.  Clearly,
\[
m(\tau^{\ge n+1}\q\beta_{\le -n}\cF) \ge n+1 > m(\cF).
\]
Now, let $k = m(\D\q\beta_{\le -n}\cF)$.  By definition, $\q\beta_{\le
-n}\cF \in \D(\dgbg Xk)$.  The $t$-structure $(\D(\dgbg Xk), \D(\dgbl
Xk))$, which is dual to (a shift of) the standard $t$-structure, is an
example of a perverse coherent $t$-structure~\cite{bez:pcs}, and therefore
of a staggered $t$-structure in the sense of~\cite{a}, so
Corollary~\ref{cor:stag-trunc} tells us that $\D(\dgbg Xk)$ is stable under
$\tau^{\ge n+1}$.  In particular, $\tau^{\ge n+1}\q\beta_{\le -n}\cF \in
\D(\dgbg Xk)$, so
\[
m(\D\tau^{\ge n+1}\q\beta_{\le -n}\cF) \ge k = m(\D\q\beta_{\le -n}\cF) \ge
m(\D\cF).
\]
We conclude that if $m(\cF) = n$, then $\mu(\tau^{\ge n+1}\beta_{\le
-n}\cF) < \mu(\cF)$.  Thus, the hypotheses of Theorem~\ref{thm:stag-gen}
are satisfied.
\end{proof}

\begin{rmk}\label{rmk:ineq2}
If $G$ does not act with finitely many orbits, or if
the $s$-structure is not recessed, Remark~\ref{rmk:ineq} tells us that
$(\{\q\dsl Xw\}, \{\q\dsg Xw\})_{w\in\Z}$ still constitutes a baric
structure if we require $q$ to be monotone and comonotone.   The proof of
Theorem~\ref{thm:stag-coh} goes through in this setting.  However, the
conditions imposed on $q$ are more restrictive than the conditions imposed
on perversity functions in~\cite{a}, so in this case, the theorem we
obtain 
is actually just a special case of~\cite[Theorem~7.4]{a}.  Similar remarks
apply to Theorems~\ref{thm:stag-dual} and~\ref{thm:stag-fl} below; {\it
cf.}~\cite[Theorems~9.7 and~9.9]{a}.
\end{rmk}

Next, we study how the duality functor $\D$ interacts with the staggered
$t$-structure.  Let $j: U \hto X$ be an open subscheme.  The following
functions defined in terms of $q: X^\gen \to \Z$ will be useful in the
sequel.
\begin{equation}\label{eqn:qpm}
\ufl q(x) =
\begin{cases}
q(x) & \text{if $x \in U^\gen$,} \\
q(x) - 1 & \text{if $x \notin U^\gen$,}
\end{cases}
\qquad
\ush q(x) =
\begin{cases}
q(x) & \text{if $x \in U^\gen$,} \\
q(x) + 1 & \text{if $x \notin U^\gen$.}
\end{cases}
\end{equation}

\begin{lem}\label{lem:stag-res}
Let $j: U \hto X$ be the inclusion of an open subscheme, and $i: Z \hto X$
the inclusion of a closed subscheme.  Then:
\begin{enumerate}
\item $j^*$ takes $\uq\dgml X0$ to $\uq\dgml U0$ and $\uq\dgpg X0$ to
$\uq\dgpg U0$.
\item $Li^*$ takes $\uq\dgml X0$ to $\uq\dgml Z0$.
\item $Ri^!$ takes $\uq\dgpg X0$ to $\uq\dgpg Z0$.
\item $i_*$ takes $\uq\dgml Z0$ to $\uq\dgml X0$ and $\uq\dgpg Z0$ to
$\uq\dgpg X0$.
\end{enumerate}
\end{lem}
\begin{proof}
We will prove the parts of this lemma in the order~(2), (4), (3), (1).

(2)~First, suppose $\cF \in \uq\dgbl X0$ is concentrated in a single
degree, say $\cF \cong h^k(\cF)[-k]$.  Then $\cF \in \q\dsl X{-k}$, so by
Lemma~\ref{lem:baric-res}, $Li^*\cF \in \q\dsml Z{-k}$.  We also clearly
have $Li^*\cF \in \dgml Zk$, so it follows that $Li^*\cF \in \uq\dgml Z0$. 
Next, for general $\cF \in \uq\dgbl X0$, induction on the number of nonzero
cohomology sheaves (together with the fact that $\uq\dgml Z0$ is stable
under extensions) allows us to reduce to the case already considered, and
we conclude that $Li^*$ takes $\uq\dgbl X0$ to $\uq\dgml Z0$.  Finally, if
$\cF \in \uq\dgml X0$, consider the distinguished triangle
\[
Li^*\tau^{\le k-1}\cF \to Li^*\cF \to Li^*\tau^{\ge k}\cF \to.
\]
Since $\tau^{\ge k}\cF \in \uq\dgbl X0$, we know that $Li^*\tau^{\ge k}\cF
\in \uq\dgml Z0$.  Moreover, we see from the long exact cohomology sequence
associated to this distinguished triangle $h^k(Li^*\cF) \cong
h^k(Li^*\tau^{\ge k}\cF) \in \q\dsl Z{-k}$.  Thus, $Li^*\cF \in \uq\dgml
Z0$, as desired.

(4)~Because $i_*$ is $t$-exact (with respect to the standard
$t$-structure), and because $\uq\dgml Z0$ and $\uq\dgml X0$ are defined by
conditions on cohomology sheaves, it follows from Lemma~\ref{lem:baric-res}
that $i_*$ takes $\uq\dgml Z0$ to $\uq\dgml X0$.  (The same argument shows
that $j^*$ takes $\uq\dgml X0$ to $\uq\dgml U0$.)  On the other hand, if
$\cF \in \uq\dgpg Z0$, then for any $\cG \in \uq\dgml X0$, we have
\[
\Hom(\cG[1], i_*\cF) \cong \Hom(Li^*\cG[1], \cF) = 0,
\]
where, in the last step, we have used the fact that $Li^*\cG \in \uq\dgml
Z0$.  Thus, $i_*\cF \in \uq\dgpg X0$.

(3)~Let $\cF \in \uq\dgpg X0$.  For any $\cG \in \uq\dgml Z0$, we have
\[
\Hom(\cG[1], Ri^!\cF) \cong \Hom(i_*\cG[1], \cF) = 0.
\]
Here, we have used the fact that $i_*\cG \in \uq\dgml X0$.  Thus, $Ri^!\cF
\in \uq\dgpg Z0$.

(1)~It was observed in the proof of part~(4) that $j^*$ takes $\uq\dgml X0$
to $\uq\dgml U0$.  Next, suppose $\cF \in \uq\dgpg X0$.  To show that
$j^*\cF \in \uq\dgpg U0$, it suffices, as noted at the beginning of this
section, to show that $\Hom(\cG[1],\cF) = 0$ for all $\cG \in \uq\dgbl
U0$.  But since $\uq\dgbl U0$ is stable under $\tau^{\ge n}$ and $\tau^{\le
n}$, we can further reduce to showing that $\Hom(\cG[1],\cF) = 0$ whenever
$\cG$ is an object of $\uq\dgbl U0$ that is concentrated in a single
degree.  Suppose $\cG \cong h^k(\cG)[-k]$.  We then have $h^k(\cG) \cong
\cG[k] \in \q\cgl U{-k}$.  Let $\ufl q: X^\gen \to \Z$ be as
in~\eqref{eqn:qpm}.  Then of course $\q\cgl U{-k} = \flatql\cgl U{-k}$, and
by Lemma~\ref{lem:q-ext}, $\cG[k]$ may be extended to a sheaf $\cG' \in
\flatql\cgl X{-k}$.  Consider the exact sequence
\[
\Hom(\cG'[-k+1],\cF) \to \Hom(j^*\cG'[-k+1],j^*\cF) \to
\lim_{\substack{\to \\ Z'}} \Hom(L\kappa^*_{Z'}\cG'[-k],
R\kappa^!_{Z'}\cF),
\]
where $\kappa_{Z'}: Z' \hto X$ ranges over all closed subscheme structures
on the complement of $U$.  We clearly have that $\cG'[-k] \in \flatqu\dgbl
X0 \subset \uq\dgbl X0$, so $\Hom(\cG'[-k+1],\cF) = 0$.  We already know
that $R\kappa^!_{Z'}\cF \in \uq\dgpg {Z'}0$, and that
$L\kappa^*_{Z'}\cG'[-k] \in \flatqu\dgml {Z'}0 = \uq\dgml {Z'}{-1}$, so the
last term above vanishes as well.  Therefore, $\Hom(j^*\cG'[-k+1], j^*\cF)
\cong \Hom(\cG[1],j^*\cF) = 0$, so $j^*\cF \in \uq\dgpg U0$, as desired.
\end{proof}

\begin{prop}
We have $\D(\uq\dgml X0) = \barq\dgpg X0$.
\end{prop}
\begin{proof}
We proceed by noetherian induction, and assume the result is known for all proper closed subschemes of $X$.  To show that $\D(\uq\dgml X0) \subset \barq\dgpg X0$, let us begin by considering the special case where $\cF \in \uq\dgml X0$ is concentrated in a single degree, say $\cF \cong h^k(\cF)[-k]$.  Then $\cF \in \q\dsl X{-k}$, so $\D\cF \in \cheq \dsg Xk$.  Choose an open orbit $C \subset X$.  By~\cite[Lemma~6.6]{a}, $\D\cF|_C$ is concentrated in a single degree, {\it viz.}, in degree $\cod \overline C -k$.  We claim that $\D\cF|_C \in \barq\dgpg C0$.  To prove this, it suffices to show that if $\cG \in \barq\dgbl C0$, then $\Hom(\cG[1],\D\cF|_C) = 0$.  Consider the exact sequence
\begin{multline*}
\Hom((\tau^{\ge \cod \overline C -k +1}\cG)[1],\D\cF|_C) \to
\Hom(\cG[1],\D\cF|_C) \to \\
\Hom((\tau^{\le \cod \overline C - k}\cG)[1],\D\cF|_C).
\end{multline*}
The last term clearly vanishes because $\D\cF|_C \in \dgbg C{\cod X-k}$.  On the other hand, note that $\tau^{\ge \cod \overline C -k +1}\cG \in \lbarq\dsl C{k-\cod \overline C -1}$.  Over the single orbit $C$, the functions $\bar q$ and $\dualq$ differ simply by the constant $\cod \overline C$, and thus $\lbarq\dsl C{k-\cod \overline C-1} = \cheq\dsl C{k-1}$.  Since $\D\cF|_C \in \cheq \dsg Ck$, the first term above vanishes, and hence so does the middle term.

We have shown that $\D\cF|_C \in \barq\dgpg C0$.  Next, let $\cG \in \barq\dgbl X0$, and consider the exact sequence
\[
\lim_{\substack{\to \\ Z'}} \Hom(Li^*_{Z'}\cG[1], Ri^!_{Z'}\D\cF) \to
\Hom(\cG[1],\D\cF) \to \Hom(\cG[1]|_C,\cD\cF|_C),
\]
where $i_{Z'}:Z' \hto X$ ranges over all closed subscheme structures on the complement of $C$.  We have just seen that the last term vanishes.  Also, $Li^*_{Z'}\cG \in \barq\dgml {Z'}0$ and $Ri^!_{Z'}\D\cF \cong \D(Li^*_{Z'}\cF) \in \barq\dgpg {Z'}0$ by Lemma~\ref{lem:stag-res} and the inductive assumption, so the first term vanishes as well.  Therefore, the middle term vanishes, and we conclude that for $\cF \in \uq\dgml X0$ concentrated in a single degree, we have$\D\cF \in \barq\dgbg X0$.  It follows by induction on the number of nonzero cohomology sheaves that $\D$ also takes all objects of the bounded category $\uq\dgbl X0$ to $\barq\dgbg X0$.  

Finally, let us consider a general object $\cF \in \uq\dgml X0$.  We wish to show that $\Hom(\cG[1], \D\cF) = 0$ for all $\cG \in
\barq\dgbl X0$.  By the previous paragraph, $\D\cG \in \uq\dgbg X0 \subset
\uq\dgpg X0$, so
\[
\Hom(\cG[1],\D\cF) \cong \Hom(\cF, \D(\cG[1])) \cong \Hom(\cF[1],\D\cG)
= 0.
\]
Thus, $\D(\uq\dgml X0) \subset \barq\dgpg X0$.

The argument for the opposite inclusion is similar, and we again use noetherian induction, but we cannot begin with the case of an object concentrated in one degree, since $\barq\dgpg X0$ is not stable under the standard truncation functors.  The bounded category $\barq\dgbg X0$ is, however, stable under the baric truncation functors $\lbarq\beta_{\le k}$ and $\lbarq\beta_{\ge k}$.  Suppose, then, that $\cF \in \barq\dgbg X0$ is ``baric-pure'': that is, $\cF \in \lbarq\dsl Xk \cap \lbarq\dsg Xk$ for some $k$.  If we prove that $\D\cF \in \uq\dgbl X0$, then it will follow by induction on ``baric length'' that $\D$ takes all objects of $\barq\dgbg X0$ to $\uq\dgbl X0$.

The assumptions on $\cF$ imply that $\cF \in \dgbg X{-k}$.  Once again, let $C \subset X$ be an open orbit.  It follows from~\cite[Lemma~6.6]{a} that $\D\cF|_C \in \dgbl C{\cod \overline C +k}$.  We also know that $\D\cF \in {}_{\Hat{\Bar q}}\dsl X{-k}$, where
\[
\Hat{\Bar q}(C) = \alt \overline C - (\alt \overline C + \cod \overline C - q(C)) = q(C) - \cod \overline C.
\]
In particular, ${}_{\Hat{\Bar q}}\dsl C{-k} = \q\dsl C{-\cod \overline C-k}$.  Since $\D\cF|_C \in \dgbl C{\cod \overline C +k}$, we see that $\D\cF|_C \in \uq\dgbl C0$.

To show that $\D\cF \in \uq\dgbl X0$, it suffices, by
Proposition~\ref{prop:compat}, to show that $\Hom(\D\cF[1],\cG) = 0$ for
all $\cG \in \uq\dgbg X0$.  Consider the exact sequence
\[
\lim_{\substack{\to \\ Z'}} \Hom(Li^*_{Z'}\D\cF[1], Ri^!_{Z'}\cG) \to
\Hom(\D\cF[1],\cG) \to \Hom(\D\cF[1]|_C,\cG|_C),
\]
where $i_{Z'}: Z' \hto X$ ranges over all closed subscheme structures on
the complement of $C$.  The last term above vanishes because $\D\cF|_C \in \uq\dgbl C0$.  We also have $Li^*_{Z'}\D\cF \cong \D(Ri^!_{Z'}\cF) \in \uq\dgml {Z'}0$ and $Ri^!_{Z'}\cG \in \uq\dgpg {Z'}0$ by Lemma~\ref{lem:stag-res} and the inductive assumption.  Hence, the first term in the sequence above vanishes, so the middle
term vanishes as well, and we conclude that $\D\cF \in \uq\dgbl X0$.  Thus,
$\D(\barq\dgbg X0) \subset \uq\dgbl X0$.

Finally, we must consider a general object $\cF \in \barq\dgpg X0$. 
Showing that $\D\cF \in \uq\dgml X0$ is equivalent to showing that
$\tau^{\ge k}\D\cF \in \uq\dgbl X0$ for all $k$.  If the latter condition
fails for some $k$, then there exists an object $\cG \in \uq\dgbg X1$ such
that $\Hom(\tau^{\ge k}\D\cF, \cG) \ne 0$.  By replacing $k$ by a smaller
integer if necessary, we may assume that $\cG \in \dgbg Xk$.  We then have
\[
\Hom(\tau^{\ge k}\D\cF, \cG) \cong \Hom(\D\cF,\cG) \cong \Hom(\D\cG,\cF)
\ne 0.
\]
By exchanging the roles of $q$ and $\bar q$ in the previous paragraph, we
see that $\D\cG \in \barq\dgbl X{-1}$, but this contradicts the fact that
$\cF \in \barq\dgpg X0$.  Therefore, $\D\cF \in \uq\dgml X0$, and
$\D(\barq\dgpg X0) = \uq\dgml X0$, as desired.
\end{proof}

The next theorem follows immediately from the last proposition.

\begin{thm}\label{thm:stag-dual}
The dual of the staggered $t$-structure $(\uq\dgbl X0, \uq\dgbg X0)$ is the
staggered $t$-structure $(\barq\dgbl X0, \barq\dgbg X0)$.  In particular,
in the case where every orbit $C \subset X$ has even staggered codimension,
and $q$ is the function $q(C) = \half\scod C$, the $t$-structure $(\uq\dgbl
X0, \uq\dgbg X0)$ is self-dual.\qed
\end{thm}

We conclude with a study of simple objects in $\uq\cM(X)$.  The statements
below and their proofs are very similar to those
in~\cite[Section~3.2]{bez:pcs} or~\cite[Section~9]{a}, and most details of
the proofs will be omitted.  Instead, each statement is followed by brief
remarks clarifying the relationship to statements in~\cite{bez:pcs}
or~\cite{a}.

\begin{prop}\label{prop:ic}
Let $j: U \hto X$ be a dense open subscheme.  Given a function $q: X^\gen
\to \Z$, define $\ufl q, \ush q: X^\gen \to \Z$ as in~\eqref{eqn:qpm}, and
define a full subcategory $\uq\cM^{!*}(X) \subset \uq\cM(X)$ by
$\uq\cM^{!*}(X) = \flatqu\dgbl X0 \cap \sharpqu\dgbg X0$.  The functor
$j^*$ induces an equivalence of categories $\uq\cM^{!*}(X) \to \uq\cM(U)$. 
Moreover, objects of $\uq\cM^{!*}(X)$ have no subobjects or quotients in
$\uq\cM(X)$ that are supported on $X \ssm U$.
\end{prop}

\begin{proof}[Remarks on proof]
This statement corresponds to~\cite[Theorem~2]{bez:pcs}
or~\cite[Proposition~9.2]{a}, but both those statements impose a condition
on the function $q$ (denoted $p$ in {\it loc.~cit.}) that is not imposed
here.  The reason is that the proof requires that the categories
$(\flatqu\dgbl X0, \flatqu\dgbg X0)$ and $(\sharpqu\dgbl X0, \sharpqu\dgbg
X0)$ associated to $\ufl q$ and $\ush q$ (denoted $p^+$ and $p^-$ in {\it
loc.~cit.}) actually constitute $t$-structures.  In the present paper,
Theorem~\ref{thm:stag-coh} tells us that this is the case with no
assumptions, whereas in both~\cite{bez:pcs} and~\cite{a}, the $t$-structure
is constructed only for $p$ obeying certain inequalities.
\end{proof}

\begin{defn}
The inverse equivalence to that of the preceding proposition, denoted
$j_{!*}: \uq\cM(U) \to \uq\cM^{!*}(X)$, is known as the
\emph{intermediate-extension functor}.
\end{defn}

\begin{defn}\label{defn:ic}
Let $Y$ be a locally closed subscheme of $X$.  Let $h: Y \hto \overline Y$
and $\kappa: \overline Y \hto X$ denote the inclusion maps.  For any $\cF
\in \uq\cM(Y)$, we define an object of $\uq\cM(X)$ by
\[
\cIC(\overline Y, \cF) = \kappa_*(h_{!*}\cF).
\]
This is called the \emph{(staggered) intersection cohomology complex}
associated to $\cF$.
\end{defn}

Recall that the \emph{step} of a coherent sheaf is defined to be the unique
integer $w$ (if such an integer exists) such that the sheaf belongs $\cgl
Xw \cap \cgg Xw$.  An irreducible vector bundle on an orbit always has a
well-defined step.

\begin{prop}\label{prop:simple}
Let $\cF \in \uq\cM(X)$.  $\cF$ is a simple object if and only if $\cF
\cong \cIC(\overline C, \cL[-q(C)+\step \cL])$ for some orbit $C \subset X$
and some irreducible vector bundle $\cL \in \cg C$.
\end{prop}
\begin{proof}[Remarks on proof]
This statement is analogous to~\cite[Corollary~4]{bez:pcs} and
to~\cite[Theorem~9.7]{a}.  The main difference is that in~\cite{a}, $\cF$
is assumed at the outset to be supported on (a possibly nonreduced
subscheme structure on) the closure of one orbit.  (The statement
of~\cite[Theorem~9.7]{a} also imposes conditions on $q$, but those are
unnecessary here for reasons explained in the remarks following
Proposition~\ref{prop:ic}.)  In~\cite{bez:pcs}, it is shown that a simple
object must be supported on an orbit closure using Rosenlicht's Theorem,
but that argument cannot be used here for the reasons given
in~\cite[Remark~9.8]{a}.

To reduce this statement to one where the proof of~\cite[Theorem~9.7]{a}
can be repeated verbatim, we must show by other means that the support of a
simple object is an ({\it a priori} possibly nonreduced) orbit closure. 
Since $X$ is assumed to consist of finitely many $G$-orbits, it suffices to
show that the support of a simple object is irreducible.  Let $\kappa: X'
\hto X$ be the scheme-theoretic support of $\cF$; that is, $\cF \cong
\kappa_*\cF'$, and the restriction of $\cF'$ to any open subscheme of $X'$
is nonzero.  Assume $X'$ is reducible; let $i:Z \hto X'$ and $i':Z' \to X'$
be proper closed subschemes such that $Z \cup Z' = X'$.  Let $U = Z \ssm (Z
\cap Z')$ and $U' = Z' \ssm (Z \cap Z')$.  Clearly, $U$ and $U'$ are
disjoint open subschemes of $X'$.  Let $V = U \cup U'$.  The natural
morphism
\[
i_*Ri^!\cF'|_V \to \cF'|_V
\]
is the inclusion of the direct summand of $\cF|_V$ supported on $U$.  In
particular, the above morphism is neither $0$ nor an isomorphism.  But it
is also the restriction to $V$ of the natural morphism
\[
\uq h^0(i_*Ri^!\cF') \to \cF',
\]
so this latter is also neither $0$ nor an isomorphism.  Therefore, $\cF'$
is not simple, and hence neither is $\cF$.
\end{proof}

\begin{thm}\label{thm:stag-fl}
$\uq\cM(X)$ is a finite-length category.
\end{thm}
\begin{proof}[Remarks on proof]
This statement and its proof are identical to those
of~\cite[Corollary~5]{bez:pcs} or of~\cite[Theorem~9.9]{a}, except that
here, as in Propositions~\ref{prop:ic} and~\ref{prop:simple}, we impose no
restrictions on $q$.
\end{proof}


\end{document}